\documentclass[12pt]{amsart}

\usepackage[T1]{fontenc}
\usepackage{amssymb}
\usepackage{bm}
\usepackage{amsmath}
 \usepackage{mathpple}
\usepackage{boondox-cal}
\usepackage{amsmath,amsfonts, amscd,amsthm, graphicx}
  \usepackage{amsrefs}
\usepackage[all]{xy}
\usepackage{color}
\usepackage{tikz-cd}

\usepackage{hyperref}

\def\XXint#1#2#3{{\setbox0=\hbox{$#1{#2#3}{\int}$}
\vcenter{\hbox{$#2#3$}}\kern-.5\wd0}}

\numberwithin{equation}{section}

\newcommand{\mbf}{\mathbf}
\newcommand{\mbb}{\mathbb}
\newcommand{\mf}{\mathfrak}
\newcommand{\mc}{\mathcal}

\theoremstyle{plain}
\newtheorem{thm}{Theorem}[section]

\newtheorem{thm-defn}{Theorem/Definition}[section]
\newtheorem{lem}[thm]{Lemma}
\newtheorem{lem-defn}[thm]{Lemma/Definition}
\newtheorem{prop}[thm]{Proposition}
\newtheorem{cor}[thm]{Corollary}

\newtheorem{prop-defn}[thm]{Proposition-Definition}

\newtheorem{defn}[thm]{Definition}

\newtheorem{thm-alg}[thm]{Theorem/Algorithm}

\allowdisplaybreaks[4]  

\begin{document}

 \title{Instanton moduli space, stable envelopes and quantum difference equations}
  \author{Tianqing Zhu}
  \address{Yau Mathematical Sciences Center}
\email{ztq20@mails.tsinghua.edu.cn}

  \date{}

  \maketitle
\begin{abstract} In this article we prove the degeneration limit of the quantum difference equations of instanton moduli space for both algebraic one and the Okounkov-Smirnov geometric one is the Dubrovin connection for the instanton moduli space. As an application, we prove that the Maulik-Okounkov quantum algebra of Jordan type is isomorphic to the quantum toroidal algebra $U_{q_1,q_2}(\hat{\hat{\mf{gl}}}_1)$ up to some centres.
\end{abstract}

\tableofcontents
\section{\textbf{Introduction}}
\subsection{Quantum difference and differential equations of the instanton moduli space}

In this paper we compare the degeneration limit of both algebraic and geometric quantum difference equation of the instanton moduli space $M(n,r)$, i.e. the Nakajima quiver varieties of Jordan type.

The algebraic quantum difference equation for the instanton moduli space is written as follows \cite{OS22}:
\begin{equation}
\Psi(pz)=\mbf{M}_{\mc{O}(1)}(z)\Psi(z),\qquad\Psi(z)\in K_{T}(M(n,r))_{loc}\otimes_{\mbb{Q}}\mc{K}_z
\end{equation}

\begin{equation}
\begin{aligned}
\mbf{M}_{\mc{O}(1)}(z)=&\mc{O}(1)\prod^{\leftarrow}_{\substack{-1\leq a/b<0\\ b\leq n}}\mbf{B}_{a/b}(z)\\
&\mc{O}(1)\prod^{\leftarrow}_{\substack{-1\leq a/b<0\\ b\leq n}}:\exp(-\sum_{k=1}^{\infty}\frac{n_k}{1-z^{-kb}p^{ka}}P_{kb,ka}P_{-kb,-ka}):
\end{aligned}
\end{equation}

Here $P_{kb,ka}$ are the generators of the quantum toroidal algebra $U_{q_1,q_2}(\hat{\hat{\mf{gl}}}_1)$. 

The geometric quantum difference equation has the same form as the algebraic quantum difference equation, and we only replace the monodromy operators $\mbf{B}_{a/b}(z)$ by the geometric monodromy operators $\mbf{B}_{a/b}^{MO}(z)$. This is the quantum difference equation constructed via the MO quantum affine algebra $U_{q}(\hat{\mf{g}}_{Q})$ of Jordan type. 

The geometric quantum difference equation has the regular solution as the $K$-theoretic capping operator for the instanton moduli space. The computation for the capping operator of the instanton moduli space has been done in \cite{D22}. 

For the quantum equivariant cohomology of the instanton moduli space, it has a natural connection called the Dubrovin connection. It can be written as \cite{MO12}:
\begin{equation}
\nabla_{z}:=z\frac{d}{dz}-\mc{O}(1)*_{z}\in\text{End}(H_{T}^*(M(n,r)))\otimes\mbb{C}((z))
\end{equation}

Here $\mc{O}(1)*_{z}$ is the quantum multiplication of the tautological line bundle $\mc{O}(1)$:

\begin{align}
\mc{O}(1)*_z(-)=c_1(\mc{O}(1))\cup(-)+(t_1+t_2)\sum_{n>0}\frac{nz^n}{1-z^n}\beta_{-n}\beta_{n}
\end{align}

The Dubrovin connection is a meromorphic connection over a trivial bundle on $\mbb{P}^1$ with regular singularities on
\begin{align}
\text{Sing}=\{0,\infty,e^{\frac{2\pi k}{n}}\}\subset\mbb{P}^1
\end{align}

The flat section for the Dubrovin connection is important in the theory of the equivariant quantum cohomology of the instanton moduli space. It was first discovered by Okounkov and Pandharipande in \cite{OP10}\cite{OP10-2} in the case of the Hilbert scheme $\text{Hilb}_{n}(\mbb{C}^2)$ of points on $\mbb{C}^2$.

\subsection{Main result}

Our first main result is that the degeneration limit of the algebraic and geometric quantum difference equations over the instanton moduli space coincides with the Dubrovin connection.

\begin{thm}{(See Theorem \ref{degeneration-thm-instanton} \ref{main-theorem-1})}
The degeneration limit of the geometric and algebraic quantum difference equation for the instanton moduli space $M(n,r)$ over $K_{T}(M(n,r))$ coincides with the Dubrovin connection over $H_{T}^*(M(n,r))$.
\end{thm}

This result leads to the concrete description of the monodromy representation at based point $0^+\in\mbb{P}^1$ for the Dubrovin connection of the equivariant cohomology of the instanton moduli space:

\begin{thm}{(See Theorem \ref{monodromy-representation})}
The monodromy representation for the Dubrovin connection of the instanton moduli space is generated by $\mbf{m}((1\otimes S_{\frac{a}{b}})(R_{\frac{a}{b}}^{-})^{-1})\in\mc{B}_{\frac{a}{b}}\subset U_{q_1,q_2}(\hat{\hat{\mf{gl}}}_1)$. 
\end{thm}

The connection to the Okounkov-Smirnov's quantum difference equations lie in the following theorem \cite{OS22}:
\begin{thm}{(See Theorem \ref{geometric-monodromy})}
The monodromy representation for the Dubrovin connection of the instanton moduli space is generated by $\mbf{m}((1\otimes S_{\frac{a}{b}})(R_{\frac{a}{b}}^{MO,-})^{-1})\in U_{q}^{MO}(\mf{g}_{\frac{a}{b}})\subset U_{q}^{MO}(\hat{\mf{g}}_{Q})$.
\end{thm}

Combining these results we can conclude the isomorphism of the MO quantum affine algebra of Jordan type and the quantum toroidal algebra $U_{q_1,q_2}(\hat{\hat{\mf{gl}}}_1)$:

\begin{thm}{(See Theorem \ref{Isomorphism-theorem})}
The MO quantum affine algebra $U_{q}^{MO}(\hat{\mf{g}}_{Q})$ of Jordan quiver type is isomorphic to the quantum toroidal algebra:
\begin{align}
U_{q}^{MO}(\hat{\mf{g}}_{Q})\cong U_{q_1,q_2}(\hat{\hat{\mf{gl}}}_1)
\end{align}
up to some centres.
\end{thm}

The proof of these theorems follows the logic of the proof in \cite{Z24}\cite{Z24-2}. The only difference is that the degeneration limit for the generators in $U_{q_1,q_2}(\hat{\hat{\mf{gl}}}_1)$ is much more "degenerate", thus we need to use more insights of the shuffle formula for these generators.

There are many applications of the main theorem of the paper. One of the most important applications is in the theory of the enumerative geometry for the instanton moduli space, it has been usually assumed that the the MO quantum affine algebra $U_{q}^{MO}(\hat{\mf{g}}_{Q})$ of Jordan quiver type is isomorphic to the quantum toroidal algebra. See for example \cite{S16}\cite{S21}\cite{AS24}\cite{D22}. The paper gives the affirmative answer that we can use the quantum toroidal algebra to carry out the computation in the enumerative geometry of the instanton moduli space. 

Another important ingredient in the algebraic quantum difference equation is the connection matrix \ref{connection-matrix}. In the general theory of the difference equations, the connection matrix is usually written in terms of the infinite product. Since the regular part of the solution is written as the capping operators, this implies that the connection matrix can be written as the elliptic stable envelopes, and one can use the formula in \cite{D23} to compute the connection matrix explicitly.

\subsection{Outline of the paper}
The paper is organised as follows: In section two we introduce the quantum toroidal algebra $U_{q,t}(\hat{\hat{\mf{gl}}}_1)$, its slope subalgebra and the slope factorisation. We also introduce the wall subalgebra in the slope subalgebra, and the quantum toroidal algebra action on the equivariant $K$-theory of the instanton moduli space. 

In section three we introduce the Maulik-Okounkov quantum affine algebra and its corresponding Maulik-Okounkov wall subalgebra, and we will explain how to get the degeneration limit, or cohomological limit, of the MO wall subalgebra, to the Lie subalgebra of the MO Lie algebra.

In section four and five we introduce the Dubrovin connection, algebraic and geometric quantum difference equation for the instanton moduli space

In section six and seven we show that the degeneration limit of both algebraic and geometric quantum difference equations is the Dubrovin connection. Using this fact we will compute the monodromy representation of the Dubrovin connection in terms of the wall $R$-matrices of the MO wall subalgebra.

In section eight we prove the isomorphism of the MO quantum affine algebra of Jordan type and the quantum toroidal algebra.

\subsection*{Acknowledgements}
The author would like to thank Andrei Negu\c{t}, Andrei Okounkov, Andrey Smirnov, Hunter Dinkins and Nicolai Reshetikhin for their helpful discussions on shuffle algebras, quantum difference equations and qKZ equations. The author is supported by the international collaboration grant BMSTC and ACZSP (Grant no. Z221100002722017). Part of this work was done while the author was visiting Department of Mathematics at Columbia University. The author thanks for their hospitality and provisio of excellent working environment.

\section{\textbf{Quantum toroidal algebra}}
\subsection{$q$-Heisenberg algebra}
The $q$-Heisenberg algebra is defined as:
\begin{align}
U_{q}(\hat{\mf{gl}}_1)=\mbb{Q}(q_1,q_2)\langle P_n,c^{\pm1}\rangle
\end{align}
with $c$ central and generators satisfy the following relation:
\begin{align}
[P_n,P_{n'}]=\delta^{0,n+n'}\frac{n(q_1^{n/2}-q_{1}^{-n/2})(q_{2}^{n/2}-q_{2}^{-n/2})}{(q^{-n/2}-q^{n/2})}(c^n-c^{-n})
\end{align}
It is a bialgebra with resepct to the coproduct determined by:
\begin{align}\label{coproduct-quantum-heisenberg}
&\Delta(c)=c\otimes c\\
&\Delta(P_{n})=
\begin{cases}
P_n\otimes1+c^n\otimes P_n&n>0\\
P_n\otimes c^n+1\otimes P_n&n<0
\end{cases}
\end{align}

There is a bialgebra pairing:
\begin{align}
\langle-,-\rangle:U_{q}^{\geq}(\hat{\mf{gl}}_1)\hat{\otimes}U_{q}^{\leq}(\hat{\mf{gl}}_1)\rightarrow\mbb{Q}(q_1,q_2)
\end{align}
determined by:
\begin{align}
&\langle c,-\rangle=\langle-,c\rangle=\epsilon(-)\\
&\langle P_{n},P_{-n'}\rangle=\delta^{n}_{n'}\frac{n(q^{n/2}_1-q^{-n/2}_1)(q_2^{n/2}-q_{2}^{-n/2})}{q^{n/2}-q^{-n/2}}
\end{align}

If we take $\overline{n}=(n_1\geq\cdots\geq n_t)$ goes over partitions, and consider the products:
\begin{align}
P_{\pm\overline{n}}=P_{\pm n_1}\cdots P_{\pm n_t}
\end{align}
It gives orthogonal bases with respect to the bialgebra pairing such that:
\begin{align}
\langle P_{\overline{n}},P_{-\overline{n}'}\rangle=\delta^{\overline{n}}_{\overline{n}'}z_{\overline{n}}
\end{align}

where:
\begin{align}
z_{\overline{n}}=\overline{n}!\prod_{i=1}^{t}\frac{n_i(q_1^{n_i/2}-q_1^{-n_i/2})(q_{2}^{n_i/2}-q_{2}^{-n_i/2})}{q^{n_i/2}-q^{-n_i/2}}
\end{align}

Now the universal $R$-matrix (modulo the central part) of the $q$-Heisenberg algebra is given by:
\begin{align}
R:=\sum_{\overline{n}\text{ partition}}\frac{P_{\overline{n}}\otimes P_{-\overline{n}}}{z_{\overline{n}}}=\exp[\sum_{n=1}^{\infty}\frac{P_{n}\otimes P_{-n}}{n}\frac{(q^{n/2}-q^{-n/2})}{(q_1^{n/2}-q^{-n/2}_1)(q_2^{n/2}-q^{-n/2}_2)}]
\end{align}

\subsection{Quantum toroidal algebra $U_{q_1,q_2}(\hat{\hat{\mf{gl}}}_1)$ and shuffle algebra realization}

Consider the rational function:
\begin{align}
\zeta(x)=\frac{[xq_1][xq_2]}{[x][xq]}=\frac{(1-xq_1)(1-xq_2)}{(1-x)(1-xq)}
\end{align}
Here $[x]=x^{1/2}-x^{-1/2}$ is the quantum number.

The quantum toroidal algebra $U_{q_1,q_2}(\hat{\hat{\mf{gl}}}_{1})$ is defined as:
\begin{align}
U_{q_1,q_2}(\hat{\hat{\mf{gl}}}_{1})=\mbb{Q}(q_1,q_2)\langle e_k,f_k,h_m,c_1^{\pm1},c_2^{\pm1}\rangle_{k\in\mbb{Z},m\in\mbb{Z}-\{0\}}
\end{align}
with $c_1$ and $c_2$ central and satisfying the following relation:
\begin{align}
&[h_m,h_{m'}]=\frac{\delta^0_{m+m'}m(c_2^m-c^{-m}_2)}{(q_1^{m/2}-q^{-m/2}_1)(q_2^{m/2}-q_{2}^{-m/2})(q^{-m/2}-q^{m/2})}\\
&[h_{m},e_{k}]=e_{k+m}\cdot\begin{cases}1&m>0\\-c_2^m&m<0\end{cases},\qquad [h_m,f_k]=f_{k+m}\cdot\begin{cases}1&m<0\\-c_2^m&m>0\end{cases}\\
&e(z)e(w)\zeta(\frac{z}{w})=e(w)e(z)\zeta(\frac{w}{z}),\qquad f(z)f(w)\zeta(\frac{w}{z})=f(w)f(z)\zeta(\frac{z}{w})\\
&[e_{k},f_{k'}]=\frac{(q_1^{1/2}-q_1^{-1/2})(q_{2}^{1/2}-q_{2}^{-1/2})}{(q^{-1/2}-q^{1/2})}(\psi_{k+k'}c_{1}c_{2}^{-k'}\theta(k+k')-\psi_{k+k'}c_1^{-1}c_2^{-k}\theta(-k-k'))
\end{align}

Here:
\begin{align}
&e(z)=\sum_{k\in\mbb{Z}}\frac{e_k}{z^k},\qquad f(z)=\sum_{k\in\mbb{Z}}\frac{f_k}{z^k}\\
&\psi_{\pm}(z):=\sum_{m=0}^{\infty}\psi_{\pm m}x^m=\psi_{0}^{\pm1}\exp[\sum_{m=1}^{\infty}\frac{h_{\pm m}}{m}x^m(q_1^{m/2}-q^{-m/2}_1)(q_2^{m/2}-q_2^{-m/2})(q^{m/2}-q^{-m/2})]
\end{align}

It has the standard coproduct structure defined as 
\begin{equation}
\begin{aligned}
&\Delta(h_m)=
\begin{cases}
h_m\otimes1+c_2^m\otimes h_m&\text{ if }m>0\\
h_m\otimes c_2^m+1\otimes h_m&\text{ if }m<0
\end{cases}\\
&\Delta(e_k)=e_k\otimes1+\sum_{m=0}^{\infty}c_1c_2^{k-m}\psi_m\otimes e_{k-m}\\
&\Delta(f_k)=1\otimes f_k+\sum_{m=0}^{\infty}f_{k+m}\otimes c_1^{-1}c_2^{k+m}\psi_{-m}\\
&\Delta(c_1)=c_1\otimes c_1,\qquad\Delta(c_2)=c_2\otimes c_2
\end{aligned}
\end{equation}

The quantum toroidal algebra $U_{q_1,q_2}(\ddot{\mf{gl}}_1)$ also has the elliptic Hall algebra realisation \cite{SV13}:
\begin{align}
\mc{A}=\mbb{Q}(q_1,q_2)\langle P_{n,m},c_1^{\pm1},c_2^{\pm1}\rangle_{(n,m)\in\mbb{Z}^2\backslash\{(0,0)\}}
\end{align}
with the following relation
\begin{align}
[P_{n,m},P_{n',m'}]=\delta^0_{n+n'}\frac{d(q_1^{d/2}-q_1^{-d/2})(q_{2}^{d/2}-q_{2}^{-d/2})}{(q^{-d/2}-q^{d/2})}(c_1^nc_2^m-c_1^{-n}c_2^{-m})
\end{align}
if $nm'=n'm$ and $n>0$, with $d=\text{gcd}(m,n)$. And if $nm'>n'm$ and the triangle with vertices $(0,0),(n,m),(n+n',m+m')$ contains no lattice points inside nor on one of the edges, we have the relation:
\begin{equation}
\begin{aligned}
&[P_{n,m},P_{n',m'}]=\frac{(q_1^{d/2}-q_1^{-d/2})(q_2^{d/2}-q_2^{-d/2})}{q^{-1/2}-q^{1/2}}Q_{n+n',m+m'}\cdot\\
&\begin{cases}
c_1^{-n'}c_2^{-m'}&(n,m)\in\mbb{Z}^2_{\pm},(n',m')\in\mbb{Z}^2_{\mp},(n+n',m+m')\in\mbb{Z}^2_{\pm}\\
c_1^nc_2^m&(n,m)\in\mbb{Z}^2_{\pm},(n',m')\in\mbb{Z}^2_{\mp},(n+n',m+m')\in\mbb{Z}^2_{\mp}\\
1&\text{otherwise}
\end{cases}
\end{aligned}
\end{equation}

Here $d=\text{gcd}(n,m)\text{gcd}(n',m')$ and
\begin{align}
\sum_{k=0}^{\infty}Q_{ka,kb}x^k=\exp[\sum_{k=1}^{\infty}\frac{P_{ka,kb}}{k}x^k(q^{k/2}-q^{-k/2})]
\end{align}

The quantum toroidal algebra $U_{q_1,q_2}(\hat{\hat{\mf{gl}}}_1)$ has another realization given by the shuffle algebra.

Consider the following vector space over $\mbb{F}=\mbb{Q}(q_1,q_2)$:
\begin{align}
V=\bigoplus_{n\geq0}\mbb{F}(z_1,\cdots,z_n)^{\text{Sym}}
\end{align}

The vector space $V$ has a shuffle product structure defined as:
\begin{align}
F(z_1,\cdots,z_k)*G(z_1,\cdots,z_l)=\frac{1}{k!l!}\text{Sym}[F(z_1,\cdots,z_k)G(z_{k+1},\cdots,z_{k+l})\prod_{i=1}^{k}\prod_{j=k+1}^{k+l}\zeta(\frac{z_i}{z_j})]
\end{align}
Here $\text{Sym}$ is the symbol for the symmetrisation:
\begin{align}
\text{Sym}(F(z_1,\cdots,z_n))=\sum_{\sigma\in S_n}F(z_{\sigma(1)},\cdots,z_{\sigma(n)})
\end{align}

\begin{defn}
We define the \textbf{shuffle algebra} $\mc{S}^+\subset V$ as the set of rational symmetric functions of the following form:
\begin{align}
F(z_1,\cdots,z_n)=\frac{r(z_1,\cdots,z_n)}{\prod_{1\leq i\neq j\leq n}(z_i-z_jq)}
\end{align}
such that $r(z_1,\cdots,z_n)$ satisfies the wheel conditions:
\begin{align}
r(z_1,\cdots,z_n)|_{\{\frac{z_1}{z_2},\frac{z_2}{z_3},\frac{z_3}{z_1}\}=\{q_1,q_2,\frac{1}{q}\}}=0
\end{align}
\end{defn}

It was observed in \cite{N14} that we have the following isomorphism:
\begin{equation}
\begin{aligned}
&\pi_{+}:U_{q_1,q_2}(\hat{\hat{\mf{gl}}}_1)^{+}\rightarrow\mc{S},\qquad e_{k}\rightarrow z_{1}^k\\
&\pi_{-}:U_{q_1,q_2}(\hat{\hat{\mf{gl}}}_1)^{-}\rightarrow\mc{S}^{op},\qquad f_{k}\rightarrow z_{1}^k
\end{aligned}
\end{equation}

The coproduct structure can be written in terms of the shuffle algebra:
\begin{equation}\label{coproduct-quantum-toroidal-drinfeld}
\begin{aligned}
&\Delta(\psi_{\pm}(w))=\psi_{\pm}(w)\otimes \psi_{\pm}(w)\\
&\Delta(P(z_1,\cdots,z_k))=\sum_{i=0}^{k}\frac{\prod_{b>i}\psi_{\pm}(z_b)\cdot P(z_1,\cdots,z_i\otimes z_{i+1},\cdots,z_k)}{\prod_{a\leq i<b}\zeta(z_b/z_a)}
\end{aligned}
\end{equation}

The RHS means that the expand the fraction in non-negative powers of $z_a/z_b$ for $a\leq i<b$ and thus obtaining an infinite sum of monomials. Then we put all the $h_n$'s to the very left of the expression, then all powers of $z_1,\cdots,z_i$ to the left of the $\otimes$ sign, and finally all powers of $z_{i+1},\cdots,z_{k}$ to the right of the $\otimes$ sign.

The corresponding elliptic Hall algebra generators $P_{\pm k,d}$ is mapped to the following symmetric Laurent rational functions \cite{N14}:

\begin{align}
P_{k,d}=q^{\frac{k-n}{2}}\text{Sym}[\frac{\prod_{i=1}^kz_{i}^{\lfloor\frac{id}{k}\rfloor-\lfloor\frac{(i-1)d}{k}\rfloor}}{\prod_{i=1}^{k-1}(1-\frac{qz_{i+1}}{z_{i}})}\sum_{s=0}^{n-1}q^s\frac{z_{a(n-1)+1}\cdots z_{a(n-s)+1}}{z_{a(n-1)}\cdots z_{a(n-s)}}\prod_{i<j}\zeta(\frac{z_i}{z_j})]
\end{align}

Here $n=\text{gcd}(k,d)$, $a=\frac{k}{n}$.

This set of generators by elliptic Hall algebra admits the quantum toroidal algebra $U_{q,t}(\hat{\hat{\mf{gl}}}_1)$ with the following slope decomposition:
\begin{align}
U_{q,t}(\hat{\hat{\mf{gl}}}_1)=\bigotimes^{\rightarrow}_{a/b\in\mbb{Q}\sqcup\{\infty\}}\mc{B}_{a/b}
\end{align}
Each $\mc{B}_{a/b}=\mbb{Q}(q_1,q_2)\langle P_{kb,ka},c_1,c_2\rangle_{k\in\mbb{Z}}$ is isomorphic to $U_{q}(\hat{\mf{gl}}_1)$, with the coproduct $\Delta_{\frac{a}{b}}$ defined as those \ref{coproduct-quantum-heisenberg} in $U_{q}(\hat{\mf{gl}}_1)$.

This factorization preserves the Hopf pairing for each $\mc{B}_{a/b}$. This result leads to the factorization of the universal $R$-matrix of the quantum toroidal $U_{q,t}(\hat{\hat{\mf{gl}}}_1)$:
\begin{align}
R_{U_{q,t}(\hat{\hat{\mf{gl}}}_1)}=\prod_{\mu\in\mbb{Q}\sqcup\{\infty\}}^{\infty}R_{\mu}=\prod_{a/b\in\mbb{Q}\sqcup\{\infty\}}^{\rightarrow}\exp[\sum_{k=1}^{\infty}\frac{P_{ka,kb}\otimes P_{-ka,-kb}}{k}\cdot\frac{(q^{k/2}-q^{-k/2})}{(q_1^{k/2}-q_{1}^{-k/2})(q_2^{k/2}-q_2^{-k/2})}]
\end{align}
which corresponds to the coproduct $\Delta$ defined in \ref{coproduct-quantum-toroidal-drinfeld}.

One can twist the coproduct $\Delta$ via:
\begin{align}
\Delta_{(\frac{a}{b})}(a)=[\prod_{\mu\in\mbb{Q}_{>0}\sqcup\{\infty\}}^{\rightarrow}R_{\frac{a}{b}+\mu}]\cdot\Delta(a)\cdot[\prod_{\mu\in\mbb{Q}_{>0}\sqcup\{\infty\}}R_{\frac{a}{b}+\mu}]^{-1}
\end{align}

and here the corresponding universal $R$-matrix is given by:
\begin{align}
R_{U_{q,t}(\hat{\hat{\mf{gl}}}_1)}^{\frac{a}{b}}=\prod_{\mu>\frac{a}{b}}^{\leftarrow}R_{\mu}^{-}R_{\infty}\prod_{\mu\leq\frac{a}{b}}^{\rightarrow}R_{\mu}^{+}
\end{align}

Using the proof in the Proposition $2.2$ in \cite{Z23}, one can prove the following:
\begin{prop}
$\Delta_{(\frac{a}{b})}(a)$ coincides with $\Delta_{\frac{a}{b}}(a)$ if $a\in\mc{B}_{\frac{a}{b}}$.
\end{prop}

For the generators of $\mc{B}_{a/b}$, we can also use the group-like generators $E^{a/b}_{k}$:
\begin{align}
E^{a/b}_{k}=\text{Sym}[\frac{\prod_{i=1}^{kb}z_{i}^{r_{\frac{m}{n}}(i)}\prod_{i=1}^{k-1}(1-\frac{z_{in}}{qt^{i-1}z_{in+1}})}{[k]_{q_1}!(1-\frac{qz_1}{z_2})\cdots(1-\frac{qz_{N-1}}{z_{N}})}\prod_{1\leq i<j\leq N}\zeta(\frac{z_i}{z_j})]
\end{align}

Here
\begin{align}
r_{\frac{m}{n}}(i)=\lfloor\frac{id}{k}\rfloor-\lfloor\frac{(i-1)d}{k}\rfloor
\end{align}

The relation between $E^{a/b}_{k}$ and $P^{a/b}_{k}$ is similar to the relation between the elementary symmetric polynomials $e_k$ and the $k$-th power sum $p_k$ as explained in \cite{N16}. Moreover, $E^{a/b}_{k}$ has the coproduct formula given as:
\begin{align}\label{coproduct-group-elements}
\Delta_{a/b}(E^{a/b}_{k})=\sum_{0\leq l\leq k}E^{a/b}_{l}\otimes c^{bl}E^{a/b}_{k-l}
\end{align}

\subsection{Fock space representation}

We now give the action of the quantum toroidal algebra $U_{q,t}(\hat{\hat{\mf{gl}}}_1)$ on $K_{T}(M(1))$, i.e. the Fock space.

We define the space of symmetric polynomials over the coefficient field $\mbb{F}=\mbb{Q}(q_1,q_2)$ as:
\begin{align}
\text{Fock}_{q_1,q_2}=\mbb{F}[x_1,x_2,\cdots]^{\text{Sym}}=\mbb{F}[p_{-1},p_{-2},\cdots]
\end{align}
Here $p_{-k}=\sum_{i}x_{i}^k$ are the elemenetary symmetric polynomials. It is known that $\text{Fock}_{q_1,q_2}$ has a natural $\mbb{Z}$-basis given by partitions $\lambda=(\lambda_1\geq\lambda_2\geq\cdots\geq\lambda_{l+1})$ such that $p_{\lambda}=p_{1}^{k_1}p_{2}^{k_2}\cdots p_{l}^{k_{l}}$, $k_{i}=\lambda_{i+1}-\lambda_{i}$.

Over the coefficient field $\mbb{F}$, $\text{Fock}_{q_1,q_2}$ also exhibits another set of basis, which is called the Macdonald polynomial $P_{\lambda}$ with $\lambda$ the corresponding partition.

Given a certain partition $\lambda=(\lambda_1\geq\lambda_2\geq\cdots\geq\lambda_l)$, we can draw the corresponding Young diagram as a bunch of boxes with $\lambda_i$ boxes at $i$-th row. And for each box $\square\in\lambda$, we can associate a weight $\chi_{\square}\in\mbb{F}$ defined as:
\begin{align}
\chi_{\square}=uq_1^xq_2^y
\end{align}
Here $(x,y)$ is the coordinate of the bottom left corner of the box $\square$.

Given two Young diagrams $\mu,\lambda$, we shall write $\mu\subset\lambda$ if $\mu$ is contained in $\lambda$. Let $R$ be a symmetric rational function in $|\lambda\backslash\mu|$ variables, we denote:
\begin{align}
R(\lambda\backslash\mu)=R(\cdots,\chi_{\square},\cdots)_{\square\in\lambda\backslash\mu}
\end{align}
as the function $R$ evaluated at $\{\chi_{\square}\}_{\square\in\lambda\backslash\mu}$.

We denote $M(1,n):=\text{Hilb}_{n}(\mbb{C}^2)$ as the Hilbert scheme of $n$ points over $\mbb{C}^2$, it has a natural action $T:=(\mbb{C}^{\times})^2$ induced by the action on $\mbb{C}^2$. The fixed point set $\text{Hilb}_{n}(\mbb{C}^2)^{(\mbb{C}^{\times})^2}$ corresponds to the set of partitions $\lambda$ such that $|\lambda|=n$. In this case, simplifying the notation $M(1):=\sqcup_{n}M(1,n)$, the localized equivariant $K$-theory $F(u):=K_{T}(M(1))_{loc}$ has the fixed point basis denoted by $|\lambda\rangle$ with $\lambda$ a partition.

As the $U_{q,t}(\hat{\hat{\mf{gl}}}_1)$-module, there is an isomorphism between $F(u)$ and $\text{Fock}_{q_1,q_2}$ given by:
\begin{align}
|\lambda\rangle\mapsto P_{\lambda}
\end{align}
such that the fixed point basis corresponds to the Macdonald polynomial.

Now we use the shuffle algebra realization of $U_{q,t}(\hat{\hat{\mf{gl}}}_1)$ to give an explicit formula for the quantum toroidal action over $F(u)$. For the details see \cite{N14}\cite{FT11}.

Let $X\in\mc{S}^{+}$ such that $\pi^{+}(X)=R$ and $R$ is some symmetric rational function. Then we have that
\begin{align}
\langle\lambda|X|\mu\rangle=R(\lambda\backslash\mu)(\frac{(1-q_1)(1-q_2)}{1-q})^{|\lambda\backslash\mu|}\prod_{\blacksquare\in\lambda\backslash\mu}\frac{\prod_{\square\text{ o.c. of }\lambda}[\frac{\chi_{\square}}{\chi_{\blacksquare}}]}{\prod_{\square\text{ i.c. of }\lambda}[\frac{\chi_{\square}}{\chi_{\blacksquare}}]}
\end{align}

Similarly for $Y\in\mc{S}^{-}$ such that $\pi^{-}(Y)=R$, we have:
\begin{align}
\langle\mu|Y|\lambda\rangle=R(\lambda\backslash\mu)(\frac{(1-q_1)(1-q_2)}{1-q}q^{1/2})^{-|\lambda\backslash\mu|}\prod_{\blacksquare\in\lambda\backslash\mu}\frac{\prod_{\square\text{ i.c. of }\lambda}[\frac{\chi_{\square}}{q\chi_{\blacksquare}}]}{\prod_{\square\text{ o.c. of }\lambda}[\frac{\chi_{\square}}{q\chi_{\blacksquare}}]}
\end{align}

For the Cartan generators $h_{\pm d}$:
\begin{align}
\langle\mu|h_{\pm d}|\lambda\rangle=\delta^{\mu}_{\lambda}u^{\pm d}\text{sign}(d)(\frac{1}{(1-q_1^{\pm r})(1-q_2^{\pm r})}-\sum_{\square=(i,j)\in\lambda}q_1^{\pm(i-1)k}q_2^{\pm(j-1)k})
\end{align}

The central elements $c_1$ and $c_2$ are sent to the scaling operators:
\begin{align}
c_1\mapsto q,\qquad c_2\mapsto1
\end{align}

The quantum toroidal algebra $U_{q,t}(\hat{\hat{\mf{gl}}}_1)$ contains a $q$-Heisenberg subalgebra $\mc{B}_{0}=\mbb{F}\langle P_{k,0}\rangle_{k\in\mbb{Z}}$, and when $\mc{B}_{0}$ acts on $F_u=\mbb{F}[p_{-1},p_{-2},\cdots]$, the operators $P_{n,0}$ can be written as:
\begin{align}
p_{-n}=P_{-n,0},\qquad p_n=-n(1-q_1^n)(1-q_2^n)\frac{\partial}{\partial p_{-n}}=P_{n,0}
\end{align}

And the fixed point basis $|\lambda\rangle$ corresponds to the Macdonald polynomial $P_{\lambda}$.

If we do the degeneration over the variables $q=e^{\hbar(t_1+t_2)}, t=e^{\hbar(t_1-t_2)}$ such that $\hbar\rightarrow0$. The MacDonald polynomial degenerate to the Jack polynomial $J_{\lambda}(t_1,t_2)$ over the field $\mbb{Q}(t_1,t_2)$. In this case take the scaling:
\begin{align}
\alpha_{kb,ka}=\frac{(q_1-1)^{kb}(1-q_2)^{kb}}{(q_1^{k}-1)(1-q_{2}^k)}P_{kb,ka}
\end{align}

such that:
\begin{align}
[P_{-ak',-bk'},P_{ak,bk}]=\delta_{k+k'}^0k(1-q_1^k)(1-q_2^k)
\end{align}

\begin{align}
[P_{-ak',0},P_{ak,0}]=\delta_{k+k'}^0ka(1-q_1^{ka})(1-q_2^{ka})
\end{align}

\subsection{Tensor products of Fock representations and instanton moduli space}
In this subsection we review the action of quantum toroidal algebra over $Fock_{q_1,q_2}^{\otimes r}$ and $\oplus_{n}K_{T}(M(n,r))$, i.e. the equivariant $K$-theory of the instanton moduli space $M(n,r)$.

Instanton moduli space $M(n,r)$ is the moduli space of the following data: We consider the projective plane $\mbb{P}^2$ and fix a line $\mc{l}_{\infty}\in\mbb{P}^2$, and $M(n,r)$ is the moduli space of rank $r$ torsion free sheaves $\mc{F}$ on $\mbb{P}^2$ such that it is trivialised over $\mc{l}_{\infty}$, i.e.
\begin{align}
\mc{F}|_{\mc{l}_{\infty}}\cong\mc{O}_{\mc{l}_{\infty}}^{\oplus r}
\end{align}
such that $c_2(\mc{F})=n$. 

There is also another description for the instanton moduli space $M(n,r)$ via the quiver varieties. Consider the quiver representations:
\begin{align}
T^*\text{Rep}(n,r)=\text{Hom}(\mbb{C}^{n},\mbb{C}^n)^{\oplus 2}\oplus\text{Hom}(\mbb{C}^n,\mbb{C}^r)\oplus\text{Hom}(\mbb{C}^r,\mbb{C}^n)
\end{align}
It has a natural $GL_n$ action given by:
\begin{align}
g\cdot(X,Y,I,J)=(gXg^{-1},gYg^{-1},gI,Jg^{-1})
\end{align}
which is Hamiltonian with respect to the standard symplectic structure on $T^*\text{Rep}(n,r)$. It gives the moment map:
\begin{align}
\mu:T^*\text{Rep}(n,r)\rightarrow\mf{gl}_n^*,\qquad\mu(X,Y,I,J)=[X,Y]+JI
\end{align}

Now choose a stability condition $\chi:GL_n\rightarrow\mbb{C}^*$, $\chi(g)=\text{det}(g)^{\theta}$, we define the GIT quotient of the quiver representation by:
\begin{align}
M_{\theta}(n,r):=\mu^{-1}(0)//_{\theta}GL_n,\qquad\theta\in\mbb{Z}_{+}
\end{align}

The following isomorphism of symplectic varieties was constructed by Nakajima \cite{N99}:
\begin{align}
M_{\theta}(n,r)\cong M(n,r)
\end{align}

On $M_{\theta}(n,r)$, we have natural torus action $T=(\mbb{C}^*)^2\times(\mbb{C}^*)^r$, the second component $(\mbb{C}^*)^r\subset GL_r$ is the maximal torus of the framed $GL_r$ acting as:
\begin{align}
(q,t,g)\cdot(X,Y,I,J)=(q_1X,q_2Y,\sqrt{q_1q_2}Ig,\sqrt{q_1q_2}g^{-1}J)
\end{align}
Geometrically speaking, on the $M(n,r)$ side, $(\mbb{C}^*)^2$ acts as rescaling the coordinates on $\mbb{P}^2$ preserving $\mc{l}_{\infty}$, and the torus $(\mbb{C}^*)^r$ acts as left multiplication on $\mc{F}|_{\mc{l}_{\infty}}\cong\mc{O}_{\mc{l}_{\infty}}^{\oplus r}$.

We consider the equivariant K-theory $K_{T}(M(r))=\oplus_{n}K_{T}(M(n,r))$ of the instanton moduli space $M(n,r)$. It is a module over $K_{T}(pt)=\mbb{Z}[q_1^{\pm1},q_2^{\pm1},u_i^{\pm1}]_{i=1,\cdots,r}$. In this paper we consider the localized equivariant $K$-theory $K_{T}(M(r))_{loc}$ of the instanton moduli space:
\begin{align}
K_{T}(M(r))_{loc}:=K_{T}(M(r))\otimes_{K_{T}(pt)}\text{Frac}(K_T(pt))
\end{align}
where $\text{Frac}(K_{T}(pt)):=\mbb{Q}(q_1,q_2,u_i)_{i=1,\cdots,r}$. 

Via the localisation theorem, the localised K-theory $K_{T}(M(n,r))_{loc}$ is isomorphic to the localised K-theory $K_{T}(M(n,r)^T)_{loc}$ of the $T$-fixed points. A well-known fact is that the $T$-fixed points of $M(n,r)$ corresponds to the r-partitions
\begin{align}
\bm{\lambda}=(\lambda^1,\cdots,\lambda^r)
\end{align}
such that $|\bm{\lambda}|=\sum_{i=1}^r|\lambda_i|=n$.

For each box $\square$ in the $r$-partition $\bm{\lambda}$ we define the weight function:
\begin{align}
\chi_{\square}=u_kq_1^iq_2^j
\end{align}
Here $u_k$ means that the box $\square$ lives in the partition $\lambda^k$. $(i,j)$ are the coordinates of the box $\square$.

The instanton moduli space $M(n,r)$ has a tautological bundle $V$. Restricting $V$ to the fixed point $\bm{\lambda}$, we have that:
\begin{align}\label{character-formula-tauto}
V|_{\bm{\lambda}}=\sum_{\square\in\bm{\lambda}}\chi_{\square}
\end{align}

The action of $U_{q_1,q_2}(\hat{\hat{\mf{gl}}}_1)$ on $K_{T}(M(r))$ can be written as:
\begin{align}\label{fixed-point-formula}
&\langle\bm{\mu}|\psi_{\pm}(z)|\bm{\lambda}\rangle=\delta^{\bm{\mu}}_{\bm{\lambda}}\prod_{\square\in\bm{\lambda}}\frac{\zeta(\frac{z}{\chi_{\square}})}{\zeta(\frac{\chi_{\square}}{z})}\prod_{i=1}^r\frac{[\frac{u_i}{qz}]}{[\frac{z}{qu_i}]},\qquad \langle\bm{\mu}|\psi_{0}^{\pm}|\bm{\lambda}\rangle=\delta^{\bm{\mu}}_{\bm{\lambda}}q^{\pm|\bm{\lambda}|r}\\
&\langle\bm{\lambda}|X|\bm{\mu}\rangle=R(\bm{\lambda}\backslash\bm{\mu})\prod_{\blacksquare\in\bm{\lambda}\backslash\bm{\mu}}[\frac{(1-q_1)(1-q_2)}{1-q}\zeta(\frac{\chi_{\blacksquare}}{\chi_{\bm{\mu}}})\tau(q\chi_{\blacksquare})],\qquad\pi^+(X)=R\\
&\langle\bm{\mu}|Y|\bm{\lambda}\rangle=R(\bm{\lambda}\backslash\bm{\mu})\prod_{\blacksquare\in\bm{\lambda}\backslash\bm{\mu}}[(\frac{q^{r/2}(1-q_1)(1-q_2)}{(1-q)})^{-1}\zeta(\frac{\chi_{\blacksquare}}{\chi_{\bm{\mu}}})^{-1}\tau(\chi_{\blacksquare})],\qquad\pi^{-}(Y)=R
\end{align}

Here $\zeta(\frac{\chi_{\blacksquare}}{\chi_{\bm{\mu}}})=\prod_{\square\in\bm{\mu}}\zeta(\frac{\chi_{\blacksquare}}{\chi_{\square}})$, $\tau(z)=\prod_{i=1}^r[\frac{z}{u_i}]$.

As the vector space, $K_{T}(M(r))_{loc}$ is isomorphic to $K_{T}(M(1))_{loc}^{\otimes r}$. The relation between the action of $U_{q_1,q_2}(\hat{\hat{\mf{gl}}}_1)$ on $K_{T}(M(r))_{loc}$ and $K_{T}(M(1))_{loc}^{\otimes r}$ can be explained as follows.

For the maximal torus $A=(\mbb{C}^{*})^r\subset GL_r$ of the framed group. The corresponding fixed point is isomorphic to the smaller quiver varieties, i.e.
\begin{align}
M(n,r)^{A}=\sqcup_{n_1+\cdots+n_r=n}M(n_1,1)\times\cdots\times M(n_r,1)
\end{align}

We consider the slope $\infty$ stable basis map such that:
\begin{align}
\text{Stab}_{\infty}^{\pm}:K(1)^{\otimes r}\rightarrow K(r),\qquad\text{Stab}^{\pm}_{\infty}(\alpha)=(i^*)^{-1}(\alpha\cdot[N^{\mp}])
\end{align}

Here $N^{\pm}$ is the $K$-theory class of the positive and negative degree part of the normal bundle $N_{M(n,r)^{A}|M(n,r)}$ with respect to the torus $A$ group action.

We can show that the stable envelope of slope $\infty$ intertwines the quantum toroidal algebra $U_{q_1,q_2}(\hat{\hat{\mf{gl}}}_1)$ action on $K(1)^{\otimes r}$ and $K(r)$ using the proof of the Proposition III.5 in \cite{N15} :

\begin{prop}\label{intertwine-infty-coproduct}
Given arbitrary $F\in U_{q_1,q_2}(\hat{\hat{\mf{gl}}}_1)$, we have the following commutative diagram:
\begin{equation}
\begin{tikzcd}
K(M(r_1))\otimes K(M(r_2))\arrow[r,"\text{Stab}_{\infty}^{\pm}"]\arrow[d,"\Delta_{(op)}(F)"]&K(M(r_1+r_2))\arrow[d,"F"]\\
K(M(r_1))\otimes K(M(r_2))\arrow[r,"\text{Stab}_{\infty}^{\pm}"]&K(M(r_1+r_2))
\end{tikzcd}
\end{equation}
\end{prop}

\section{\textbf{Brief review of the Maulik-Okounkov quantum affine algebra}}
\subsection{Maulik-Okounkov quantum affine algebras}
Here we review the construction of the Maulik-Okounkov quantum affine algebra of the Jordan type.

Given $M(\mbf{v},\mbf{w})$ a Nakajima quiver variety and $G_{\mbf{v}}\times A_{E}$ acting on $M(\mbf{v},\mbf{w})$. Given a subtorus $T\subset G_{\mbf{v}}\times A_{E}$ in the kernel of $q$. By definition, the $K$-theoretic stable envelope is a $K$-theory class
\begin{align}
\text{Stab}_{\mc{C}}^{s}\subset K_{G}(X\times X^T)
\end{align}

such that it induces the morphism
\begin{align}
\text{Stab}_{\mc{C},s}:K_{G}(X^T)\rightarrow K_{G}(X)
\end{align}

such that if we write $X^T=\sqcup_{\alpha}F_{\alpha}$ into components:
\begin{itemize}
	\item The diagonal term is given by the structure sheaf of the attractive space:
	\begin{align}
	\text{Stab}_{\mc{C},s}|_{F_{\alpha}\times F_{\alpha}}=(-1)^{\text{rk }T_{>0}^{1/2}}(\frac{\text{det}(\mc{N}_{-})}{\text{det}T_{\neq0}^{1/2}})^{1/2}\otimes\mc{O}_{\text{Attr}}|_{F_{\alpha}\times F_{\alpha}}
	\end{align}

	\item The $T$-degree of the stable envelope has the bounding condition for $F_{\beta}\leq F_{\alpha}$:
	\begin{align}
	\text{deg}_{T}\text{Stab}_{\mc{C},s}|_{F_{\beta}\times F_{\alpha}}+\text{deg}_{T}s|_{F_{\alpha}}\subset\text{deg}_{T}\text{Stab}_{\mc{C},s}|_{F_{\beta}\times F_{\beta}}+\text{deg}_{T}s|_{F_{\beta}}
	\end{align}

	We require that for $F_{\beta}<F_{\alpha}$, the inclusion $\subset$ is strict.
\end{itemize}

The uniqueness and existence of the $K$-theoretic stable envelope was given in \cite{AO21} and \cite{O21}. In \cite{AO21}, the consturction is given by the abelinization of the quiver varieties. In \cite{O21}, the construction is given by the stratification of the complement of the attracting set, which is much more general.

The stable envelope has the factorisation property called the triangle lemma. Given a subtorus $T'\subset T$ with the corresponding chamber $\mc{C}_{T'},\mc{C}_{T}$, we have the following diagram commute:
\begin{equation}\label{triangle-lemma}
\begin{tikzcd}
K_{G}(X^T)\arrow[rr,"\text{Stab}_{\mc{C}_T,s}"]\arrow[dr,"\text{Stab}_{\mc{C}_{T/T'},s}"]&&K_{G}(X)\\
&K_{G}(X^{T'})\arrow[ur,"\text{Stab}_{\mc{C}_{T'},s}"]&
\end{tikzcd}
\end{equation}

Here we only focus on the Nakajima quiver varieties of the Jordan type $X=M(n,r)$. In this case, the space of slope parametres $\text{Pic}(X)\otimes\mbb{Q}\cong\mbb{Q}$ is a one-dimensional vector space. We choose the framing torus $\sigma:\mbb{C}^*\rightarrow A_{r}$ such that:
\begin{align}
r=a_1r_1+\cdots+a_kr_k
\end{align}
In this case the fixed point is given by:
\begin{align}
M(n,r)^{\sigma}=\bigsqcup_{n_1+\cdots+n_k=n}M(n_1,r_1)\times\cdots\times M(n_k,r_k)
\end{align}

Denote $K(r):=\oplus_{n}K_{T}(M(n,r))$, now the stable envelope is a map:
\begin{align}
\text{Stab}_{\sigma,s}:K(r_1)\otimes\cdots\otimes K(r_k)\rightarrow K(r_1+\cdots+r_k)
\end{align}
Using this we can define the geometric $R$-matrix:
\begin{align}
\mc{R}^s:=\text{Stab}^{-1}_{-\sigma,s}\circ\text{Stab}_{\sigma,s}: K(r_1)\otimes\cdots\otimes K(r_k)\rightarrow K(r_1)\otimes\cdots\otimes K(r_k)
\end{align}
The geometric $R$-matrix can be further factorised into smaller parts:
\begin{align}
\mc{R}^s:=\prod_{1\leq i<j\leq k}\mc{R}^{s}_{ij}(\frac{a_i}{a_j}),\qquad\mc{R}^{s}_{ij}(\frac{a_i}{a_j}):K(r_1)\otimes K(r_2)\rightarrow K(r_1)\otimes K(r_2)
\end{align}
Each $\mc{R}^{s}_{ij}(\frac{a_i}{a_j})$ satisfies the trigonometric Yang-Baxter equation with spectral parametres.

\begin{defn}
The Maulik-Okounkov quantum affine algebra $U_{q}^{MO}(\hat{\mf{g}}_{Q})$ is the subalgebra of $\prod_{r}\text{End}(K(r))$ generated by the matrix coefficients of $\mc{R}^{s}_{\mc{C}}$.
\end{defn}

Given an auxillary space $V_0=\bigotimes_{r}K(r)$, $V=\bigotimes_{r'}K(r')$ and a finite rank operator
\begin{align}
m(a_0)\in\text{End}(V_0)(a_0)
\end{align}
Now the element of $U_{q}^{MO}(\hat{\mf{g}}_{Q})$ is generated by the following operators:
\begin{align}
\oint_{a_0=0,\infty}\frac{da_0}{2\pi ia_0}\text{Tr}_{V_0}((1\otimes m(a_0))\mc{R}^{s}_{V,V_0}(\frac{a}{a_0}))\in\text{End}(V(a))
\end{align}

The coproduct structure, antipode map and the counit map can be defined as follows:
The coproduct $\Delta_{s}$ on $U_{q}^{MO}(\hat{\mf{g}}_{Q})$ is defined via the conjugation by $\text{Stab}_{\mc{C},s}$, i.e. for $a\in U_{q}^{MO}(\hat{\mf{g}}_{Q})$ as $a:K(r)\rightarrow K(r)$, $\Delta_{s}(a)$ is defined as:
\begin{equation}\label{coproduct-geometric}
\begin{tikzcd}
K(r_1)\otimes K(r_2)\arrow[r,"\text{Stab}_{\mc{C},s}"]&K(r_1+r_2)\arrow[r,"a"]&K(r_1+r_2)\arrow[r,"\text{Stab}_{\mc{C},s}^{-1}"]&K(r_1)\otimes K(r_2)
\end{tikzcd}
\end{equation}

For the antipode map $S_{s}$, note that we have the isomorphism of the graded vector space $V_i\cong V_i^*$, and this isomorphism induces the isomorphism:
\begin{align}
\prod_{i}\text{End}(V_i)\cong\prod_{i}\text{End}(V_i^*)
\end{align}

The antipode map $S_{s}:U_{q}^{MO}(\hat{\mf{g}}_{Q})\rightarrow U_{q}^{MO}(\hat{\mf{g}}_{Q})$ is given by:
\begin{align}
\oint_{a_0=0,\infty}\frac{da_0}{2\pi ia_0}\text{Tr}_{V_0}((1\otimes m(a_0))\mc{R}^{s}_{V,V_0}(\frac{a}{a_0}))\mapsto\oint_{a_0=0,\infty}\frac{da_0}{2\pi ia_0}\text{Tr}_{V_0}((1\otimes m(a_0))\mc{R}^{s}_{V^*,V_0}(\frac{a}{a_0}))
\end{align}

The projection of the module $V$ to the trivial module $\mbb{C}$ induce the counit map:
\begin{align}
\epsilon:U_{q}^{MO}(\hat{\mf{g}}_{Q})\rightarrow\mbb{C}
\end{align}

Since $M(0,n)$ is just a point, we denote the vector in $K_{\mbf{0},\mbf{w}}$ as $v_{\varnothing}$, and we call it the \textbf{vacuum vector}.  We define the evaluation map:
\begin{align}\label{evaluation-module}
\text{ev}:U_{q}^{MO}(\hat{\mf{g}}_{Q})\rightarrow\prod_{r}K(r),\qquad F\mapsto Fv_{\varnothing}
\end{align}

\begin{defn}
We define the \textbf{positive half of the Maulik-Okounkov quantum affine algebra} $U_{q}^{MO,+}(\hat{\mf{g}}_{Q})$ as the quotient by the kernel of the evaluation map:
\begin{align} 
U_{q}^{MO,+}(\hat{\mf{g}}_{Q}):=U_{q}^{MO}(\hat{\mf{g}}_{Q})/\text{Ker}(ev)
\end{align}
\end{defn}

Fix the stable envelope $\text{Stab}_{\sigma,m}$ and $\text{Stab}_{\sigma,\infty}$, we can have the following factorisation of $\text{Stab}_{\pm,m}$ near $u=0,\infty$:
\begin{equation}
\begin{aligned}
\text{Stab}_{\sigma,m}=&\text{Stab}_{\sigma,-\infty}\cdots\text{Stab}_{\sigma,m_{-2}}\text{Stab}_{\sigma,m_{-2}}^{-1}\text{Stab}_{\sigma,m_{-1}}\text{Stab}_{\sigma,m_{-1}}^{-1}\text{Stab}_{\sigma,m}\\
=&\text{Stab}_{\sigma,-\infty}\cdots R_{m_{-2},m_{-1}}^+R_{m_{-1},m}^+
\end{aligned}
\end{equation}

\begin{equation}
\begin{aligned}
\text{Stab}_{-\sigma,m}=&\text{Stab}_{-\sigma,\infty}\cdots\text{Stab}_{-\sigma,m_2}\text{Stab}_{-\sigma,m_2}^{-1}\text{Stab}_{-\sigma,m_1}\text{Stab}_{-\sigma,m_1}^{-1}\text{Stab}_{-\sigma,m}\\
=&\text{Stab}_{-\sigma,\infty}\cdots R_{m_2,m_1}^{-}R_{m_1,m}^{-}
\end{aligned}
\end{equation}

Here $R_{m_1,m_2}^{\pm}=\text{Stab}_{\pm\sigma,m_1}^{-1}\text{Stab}_{\pm\sigma,m_2}$ is the wall $R$-matrix. Since $m_1,m_2\in\mbb{Q}$ are rational numbers, this implies that the wall in $\mbb{Q}$ are points $w$ in $\mbb{Q}$. If $m_1$ and $m_2$ are separated by a point $w$, we denote the wall $R$-matrix by $R_{w}^{\pm}$. This gives the factorisation of the geometric $R$-matrix:
\begin{align}\label{factorisation-geometry}
\mc{R}^{s}(u)=\prod_{i<0}^{\leftarrow}R_{w_i}^{-}R_{\infty}\prod_{i\geq0}^{\leftarrow}R_{w_i}^{+}
\end{align}

For the wall set of $K_{T}(M(n,r))$, it has been calculated in \cite{S20} and \cite{D23} that it is isomorphic to:
\begin{align}
\text{Walls}(M(n,r))\cap[0,1)=\{\frac{a}{b}|0\leq a<b\leq n\}
\end{align}

Following the definition in \cite{OS22} and \cite{Z24-2}, we can define the wall subalgebra $U_{q}^{MO}(\mf{g}_{w})$ generated by $q^{\Omega}R_{w}^{\pm}$, and the positive/negative half $U_{q}^{MO,\pm}(\mf{g}_{w})$ similar to the definition in section three of \cite{Z24-2}.

We can define the graded pieces of $U_{q}^{MO}(\mf{g}_{w})$ by:
\begin{align}
U_{q}^{MO,\pm}(\mf{g}_{w})=\bigoplus_{n\in\mbb{N}}U_{q}^{MO,\pm}(\mf{g}_{w})_{\pm n}
\end{align}
with $a\in U_{q}^{MO,\pm}(\mf{g}_{w})_{\pm n}$ such that $a:K(l,r)\rightarrow K(l\pm n,r)$.

Similar proof as the Lemma $3.4$ in \cite{Z24-2} shows that each graded pieces is finite-dimensional.

It has been proved in \cite{Z24-2} that $U_{q}^{MO,+}(\hat{\mf{g}}_{Q})$ is generated by the positive half of the wall subalgebras $U_{q}^{MO,+}(\mf{g}_{w})$ for arbitrary walls $w$.

\subsection{Degeneration limit to Maulik-Okounkov Yangian}
We mainly concern the connection between the geometric $R$-matrix in both $K$-theory and cohomology theory. We denote the Chern character map $ch:K_{T}(M(\mbf{v},\mbf{w}))\rightarrow H_{T}(M(\mbf{v},\mbf{w}))$ as sending $u$ to $e^{\kappa z}$, and $q$ to $e^{\kappa\hbar}$.

\begin{prop}
If we take $u=e^{\kappa z}, q=e^{\kappa\hbar}$ and take $\kappa\rightarrow0$, we obtain that:
\begin{align}
\lim_{\kappa\rightarrow0}\mc{R}^s(e^{\kappa z})=\mc{R}(z)
\end{align}
which is the geometric $R$-matrix for the corresponding Maulik-Okounkov Yangian $Y_{\hbar}(\mf{g}_{Q})$.
\end{prop}

The cohomological geometric $R$-matrix $\mc{R}(z)$ admits the expansion:
\begin{align}
\mc{R}(z)=1+\frac{r_{Q}}{z}+O(\frac{1}{z^2})\in\text{End}(H(r_1)\otimes H(r_2))
\end{align}

We define the degeneration limit of the wall $R$-matrices as:
\begin{align}
\lim_{q\rightarrow1}\frac{R_{w}^{\pm}-1}{q-1}:=r_{w}^{\pm}
\end{align}

Using the argument in section $3.5$ in \cite{Z24-2}. We can see that $r_{w}^{\pm}+\Omega$ generates the Lie subalgebra $\mf{g}^{MO}_{w}$ of $\mf{g}_{Q}^{MO}$. The positive and negative half $r_{w}^{\pm}$ generate the positive and negative half $\mf{n}_{w}^{MO,\pm}$ of the Lie algebra $\mf{g}_{w}^{MO}$.

\section{\textbf{Dubrovin connection over the instanton moduli space}}

In this section we introduce the Dubrovin connection for the equivariant cohomology of the instanton moduli space.

Fix the instanton moduli space $M(n,r)$, it is proved in \cite{MO12} that the non-localized quantum equivariant cohomology $H_{T}^*(M(n,r))[[z]]$ is generated by $H^*_{T}(pt)[[z]]$ and the quantum multiplication of th tautological line bundle $\mc{O}(1)*_{z}:=c_1(\mc{O}(1))*_{z}$.

\begin{defn}
The \textbf{Dubrovin connection} for $M(n,r)$ is the first-order differential operator in the variable $z$:
\begin{equation}
\nabla_{z}:=z\frac{d}{dz}-\mc{O}(1)*_{z}\in\text{End}(H_{T}^*(M(n,r)))\otimes\mbb{C}((z))
\end{equation}
\end{defn}

The localised equivariant cohomology $\oplus_{n}H_T^*(M(n,r))_{loc}$ has the natural action given by the Heisenberg algebra $U(\hat{\mf{gl}}_1)=\mbb{C}[\beta_{k}]_{k\in\mbb{Z}}$ and the Heisenberg algebra $U(\hat{\mf{h}})=\mbb{C}[\alpha^{(i)}_{k}]_{k\in\mbb{Z},i=1,\cdots,r}$ of $A_{r-1}$-type. The construction is given in \cite{B00}\cite{N99}\cite{MO12}. The action of $U(\hat{\mf{gl}}_1)$ on $\oplus_{n}H_{T}^*(M(n,r))$ is given via the isomorphism $\oplus_{n}H_{T}^*(M(n,r))_{loc}\cong F^{\otimes r}$. Here $F$ is the Fock space $\mbb{Q}(t_1,t_2)[p_{-1},p_{-2},\cdots]$, and the action of $\beta_{k}$ is given by:
\begin{align}
\beta_{k}:=\Delta^{r-1}(\alpha_{k})=\sum_{i=1}^{r}1^{\otimes(i-1)}\otimes\alpha_{k}\otimes1^{\otimes(r-i)}
\end{align}

$F^{\otimes r}$ also has the action given by $U(\hat{\mf{gl}}_1)^{\otimes r}$, and it has the action of the subalgebra $U(\hat{\mf{h}})$ of the Heisenberg algebra of $A_{r-1}$-type. Under the action of $U(\hat{\mf{h}})$, we can express the quantum multiplication operator $\mc{O}(1)*_{z}$ in terms of the Heisenberg operators.

\begin{thm}\cite{MO12}\label{stable-envelope-commutative}
There is a commutative diagram:
\begin{equation}
\begin{tikzcd}
F^{\otimes r}\arrow[r,"\text{Stab}_{\mf{C}}"]\arrow[d,"Q"]&H_{T}^*(M(r))_{loc}\arrow[d,"\mc{O}(1)*_z"]\\
F^{\otimes r}\arrow[r,"\text{Stab}_{\mf{C}}"]&H_{T}^*(M(r))_{loc}
\end{tikzcd}
\end{equation}
Here $\text{Stab}_{\mf{C}}$ is the cohomological stable envelope for the instanton moduli space, which is an isomorphism.
\end{thm}

We can express the quantum multiplication operator of the tautological line bundle $\mc{O}(1)$ via the generators of the Heisenberg algebra $\alpha_{k}^{(i)}$ and $\alpha_{-k}^{(i)}$:
\begin{align}
[\alpha^{(i)}_{k},\alpha^{(j)}_{-k}]=-k\delta_{ij}
\end{align}
and up to a scaling operator we have
\begin{align}
Q=\text{Cubic}+\text{Quadratic}+\text{Purely Quantum}=Q_{cl}+\text{Purely Quantum}
\end{align}

Here $Q_{cl}$ corresponds to the cup product of the first Chern class $c_1(\mc{O}(1))$, it has the presentation by the Heisenberg algebra generators as:
\begin{align}
\text{Cubic}=-\frac{1}{2}\sum_{i=1}^{r}\sum_{n,m>0}(t_1t_2\alpha^{(i)}_{-n}\alpha^{(i)}_{-m}\alpha^{(i)}_{n+m}+\alpha^{(i)}_{-n-m}\alpha^{(i)}_{n}
\alpha^{(i)}_{m})
\end{align}
\begin{align}
\text{Quadratic}=-\sum_{i=1}^{r}\sum_{n>0}(t_1+t_2)(a_i+\frac{1-n}{2})\alpha_{-n}^{(i)}\alpha_{n}^{(i)}+\sum_{i<j}\sum_{n>0}(t_1+t_2)n\alpha_{-n}^{(j)}\alpha_{n}^{(i)}
\end{align}

and the purely quantum term is
\begin{align}
\text{Purely quantum}=(t_1+t_2)\sum_{n>0}\frac{nz^n}{1-z^n}\beta_{-n}\beta_{n}
\end{align}
where
\begin{align}
\beta_{-n}=\sum_{i=1}^r\alpha_{-n}^{(i)},\qquad\beta_{n}=\sum_{i=1}^r\alpha_n^{(i)}
\end{align}

In terms of the fixed point basis and the isomorphism $H_{T}(M(1))\cong F\cong\mbb{F}[p_{-1},p_{-2},\cdots]$, the action of $U(\hat{\mf{gl}}_1)$ on $\oplus_{n}H_{T}^*(M(n,r))$ is given as:
\begin{equation}
\begin{aligned}
\beta_{\pm k}:=\Delta^{r-1}(\alpha_{\pm k})=\Delta^{r-1}(p_{\pm k})
\end{aligned}
\end{equation}

The frist two term corresponds to the classical multiplication of $\mc{O}(1)$, and we denote them by $Q_{cl}$.

In other words, we known that the fixed point basis is the eigenvector of the operator $c_{1}(\mc{O}(1))$, and now the stable basis are the eigenvectors of $Q_{cl}$.

The Dubrovin connection has the regular singularities at $z=0,\infty$ and
\begin{align}
z^m=1,\qquad m=1,\cdots,n
\end{align}

\subsection{Fundamental solution around $z=0$ and $z=\infty$}
In this subsection we give the fundamental solution for the Dubrovin connection:
\begin{align}
z\frac{d\psi(z)}{dz}=Q(z)\psi(z),\psi(z)\in H_{T}^*(M(n,r))\otimes K(z)
\end{align}

For the fundamental solution $\psi_{0}(z)$ around $z=0$, note that $Q(0)=c_1(\mc{O}(1))\cup$, i.e. the cup product of the first Chern class of the tautological line bundle $\mc{O}(1)$. The corresponding eigenvectors are the fixed point basis $|\bm{\lambda}\rangle$. For $r=1$, the fixed point basis $|\lambda\rangle$ corresponds to the Jack polynomial $J_{\lambda}$. For the general case, the eigenvector is given by the generalised Jack polynomial $J_{\bm{\lambda}}$ defined in \cite{S14} such that $\text{Stab}_{\mc{C}}(J_{\bm{\lambda}})=|\bm{\lambda}\rangle$. In this way the regular solution of $z\frac{d}{dz}-\mc{O}(1)*_z$ is given by $\text{Stab}_{\mc{C}}(Y^{\bm{\lambda}}(z))$ with the initial condition:
\begin{align}
\text{Stab}_{\mc{C}}(Y^{\bm{\lambda}}(0))=|\bm{\lambda}\rangle
\end{align}

Thus the fundamental solution around $z=0$ can be written as:
\begin{align}
\psi_{0}(z)=Y^{\bm{\lambda}}(z)z^{c(\bm{\lambda})}
\end{align}
Here $c(\bm{\lambda})$ is the eigenvalue of $\mc{O}(1)$ under the fixed point basis $|\bm{\lambda}\rangle$ and $Y^{\bm{\lambda}}(z)\in\mbb{C}(t_1,t_2)[[z]]$. Moreover, $c_{\bm{\lambda}}$ can be expressed as $\mc{O}(1)|_{\bm{\lambda}}$, which is the product of linear functions over $t_1,t_2$.

For the fundamental solution around $z=\infty$, the multiplication operator is written as:
\begin{align}
Q(\infty)=\mc{O}(1)\cup-(t_1+t_2)\sum_{n>0}n\beta_{n}\beta_{-n}
\end{align}

Using the replacement such that $\alpha_{k}^{(i)}\mapsto\alpha_{k}^{(r-i)}$, $\alpha_{k}^{(i)}\mapsto-\alpha_{k}^{(i)}$ and $a_{i}\mapsto-a_{i}$, we can conclude that $Q(\infty)$ has the same the form as $Q(0)$. This means that the eigenvector for $Q(\infty)$ is the generalised Jack polynomial $J_{\bm{\lambda}}^*$ such that:
\begin{align}
J_{\bm{\lambda}}^*=J_{\bm{\lambda}}|_{\alpha_{k}^{(i)}\mapsto-\alpha_{k}^{(r-i)},a_{i}\mapsto-a_i}
\end{align}

With the help of this property, one could write down the fundamental solution $\psi_{\infty}(z)$ around $z=\infty$ as follows:
\begin{align}
\psi_{\infty}(z)=z^{-c(\bm{\lambda}^*)}H_{\infty}(z)
\end{align}

Here $H_{\infty}(\infty)=H^*=J^*\mbf{S}$ such that $J^*$ is the matrix with the column given by the vector $J_{\bm{\lambda}}^*$, $\mbf{S}$ is the transpose matrix exchanging the partition $\bm{\lambda}$ and its transpose $\bm{\lambda}'$.

\section{\textbf{Quantum difference equation over the instanton moduli space}}
\subsection{Algebraic quantum difference equation}
The algebraic quantum difference equation over the instanton moduli space is defined as the following:
\begin{equation}
\Psi(pz)=\mbf{M}_{\mc{O}(1)}(z)\Psi(z),\qquad\Psi(z)\in K_{T}(M(n,r))_{loc}\otimes_{\mbb{Q}}\mc{K}_z
\end{equation}

Here $\mc{K}_{z}$ stands for the fraction field of formal rational functions over the torus $\mbb{C}^{\times}$ with coordinate given as $z$. $\mbf{M}_{\mc{O}(1)}(z)$ is defined as follows:
\begin{align}\label{quantum-difference-operator}
\mbf{M}_{\mc{O}(1)}(z)=\mc{O}(1)\prod^{\leftarrow}_{\substack{-1\leq a/b<0\\ b\leq n}}:\exp(-\sum_{k=1}^{\infty}\frac{(q^{-k/2}-q^{k/2})q^{-krb/2}}{k(q_1^{k/2}-q_1^{-k/2})(q_2^{k/2}-q_2^{-k/2})}\frac{1}{1-z^{-kb}p^{ka}q^{-krb/2}}P_{kb,ka}P_{-kb,-ka}):
\end{align}

In the paper \cite{S20}, Smirnov used the simplification via omiting the term $q^{-krb/2}$ in the formula, and denote $n_k=\frac{(q^{-k/2}-q^{k/2})q^{-krb/2}}{k(q_1^{k/2}-q_1^{-k/2})(q_2^{k/2}-q_2^{-k/2})}$. In this way we have that:
\begin{align}
\mbf{M}_{\mc{O}(1)}(z)=\mc{O}(1)\prod^{\leftarrow}_{\substack{-1\leq a/b\leq0\\ b\leq n}}:\exp(-\sum_{k=1}^{\infty}\frac{n_k}{1-z^{-kb}p^{ka}}P_{kb,ka}P_{-kb,-ka}):
\end{align}

This simplification would not cause ambiguity in the calculation. Most of the results in \cite{S20} is still available in the settings of \ref{quantum-difference-operator}. In this paper we still use the notation \ref{quantum-difference-operator}.

\subsection{Geometric quantum difference equation}
The Okounkov-Smirnov geometric quantum difference equation for the instanton moduli space is the difference equation of the capping operator $\mbf{J}(u,z)\in K_{G}(M(\mbf{v},\mbf{w}))^{\otimes 2}\otimes\mbb{Q}[[z]]$ over the Kahler variable $z$:
\begin{align}
\mbf{J}(u,p^{\mc{L}}z)\mc{L}=\mbf{M}_{\mc{O}(1)}(z)\mbf{J}(u,z)
\end{align}

Here we review the construction of the geometric quantum difference operator $\mbf{M}_{\mc{L}}(z)$ given by Okounkov and Smirnov in \cite{OS22}:

Here we denote $\lambda$ as $q^{\lambda}=z$, $p=q^{\tau}$ as in \cite{OS22}, and we abbreviate $R_{w}$ as $R_{w}^{MO}$ the MO wall $R$-matrix. Using the notation, we denote $F(\lambda):=F(q^{\lambda})$ as the function $F(z)$ of the Kahler variable.
For each MO wall subalgebra $U_{q}^{MO}(\mf{g}_{w})$ there are corresponding ABRR equations:
\begin{align}
J_{w}^{+}(\lambda)q_{(1)}^{-\lambda}q^{-\Omega}R_{w}^+=\hbar_{(1)}^{-\lambda}q^{\Omega}J_{w}^+(\lambda),\qquad q^{-\Omega}R_{w}^-q_{(1)}^{-\lambda}J_{w}^-(\lambda)=J_{w}^-(\lambda)q^{\Omega}q_{(1)}^{-\lambda}
\end{align}

We define:
\begin{align}
\mbf{J}_{w}^{\pm}(\lambda)=J_{w}^{\pm}(\lambda-\tau\mc{L}_w)=J_{w}^{\pm}(zp^{-\mc{L}_w}),\qquad\kappa=-\frac{r}{2}
\end{align}

The geometric monodromy operators are defined as:
\begin{align}
\mbf{B}_{w}(\lambda)=\mbf{m}((1\otimes S_{w})(\mbf{J}^{-}_{w}(\lambda)^{-1}))|_{\lambda\rightarrow\lambda+\kappa}
\end{align}
The geometric quantum difference operator for the instanton moduli space is conjugate to the operator $\mbf{M}_{\mc{O}(1)}^{MO,s}(z)$ written as:
\begin{align}\label{geometric-doperator}
\mbf{M}_{\mc{O}(1)}^{MO,s}(z)=\mc{O}(1)\prod_{w\in[-1+s,s)}^{\leftarrow}\mbf{B}_{w}(z),\qquad\mbf{B}_{w}(z)\in U_{q}^{MO}(\mf{g}_{w})
\end{align}

In this paper we alway use $\mbf{M}_{\mc{O}(1)}^{MO,s}(z)$ as the geometric quantum difference operator. Usually we take $s=0$ for simplicity, but in this paper $s$ can be arbitrary rational number.

\section{\textbf{Degeneration from algebraic quantum difference equations to Dubrovin connections}}
In this section we compute the degeneration limit of the algebraic quantum difference equation to obtain the Dubrovin connection.

For the sake of the convenience, we set $q_1=e^{\hbar t_1},q_2=e^{\hbar t_2}, q=e^{\hbar(t_1+t_2)}$.

In the quantum toroidal algebra $U_{q,t}(\hat{\hat{\mf{gl}}}_1)$, we use the generators $P_{k,d}$ in terms of the shuffle realisation:
\begin{align}\label{formula-for-group-generators}
P_{k,d}=\text{Sym}[\frac{\prod_{i=1}^kz_{i}^{\lfloor\frac{id}{k}\rfloor-\lfloor\frac{(i-1)d}{k}\rfloor}}{\prod_{i=1}^{k-1}(1-\frac{qz_{i+1}}{z_{i}})}\sum_{s=0}^{n-1}q^s\frac{z_{a(n-1)+1}\cdots z_{a(n-s)+1}}{z_{a(n-1)}\cdots z_{a(n-s)}}\prod_{i<j}\zeta(\frac{z_i}{z_j})]
\end{align}

With $n=\text{gcd}(k,d)$, $a=k/n$. Now we consider the case $P_{ak,bk}$:
\begin{align}
P_{ak,bk}=\text{Sym}[\frac{\prod_{i=1}^{ak}z_{i}^{\lfloor\frac{ib}{a}\rfloor-\lfloor\frac{(i-1)b}{a}\rfloor}}{\prod_{i=1}^{ak-1}(1-\frac{qz_{i+1}}{z_{i}})}\sum_{s=0}^{k-1}q^s\frac{z_{a(k-1)+1}\cdots z_{a(k-s)+1}}{z_{a(k-1)}\cdots z_{a(k-s)}}\prod_{i<j}\zeta(\frac{z_i}{z_j})]
\end{align}

For the element $P_{k,0}$, we have the following expression in terms of the shuffle algebra:
\begin{lem}
\begin{align}
P_{k,0}=\text{Sym}[\frac{\omega(z_1,\cdots,z_k)}{\prod_{i=1}^{k-1}(1-\frac{qz_{i+1}}{z_{i}})}\prod_{i<j}\zeta(\frac{z_i}{z_j})]
\end{align}

Here $\omega(z_1,\cdots,z_k)$ is a Laurent polynomial over $z_1,\cdots,z_k$ with coefficents in $\mbb{Q}(q,t)$.
\end{lem}
\begin{proof}
Recall that we have the following expressions:
\begin{align}
&\sum_{d\in\mbb{Z}}\frac{P_{d,1}}{x^d}=u\cdot\exp[\sum_{n=1}^{\infty}\frac{P_{-n,0}}{nx^{-n}}]\exp[\sum_{n=1}^{\infty}\frac{P_{n,0}}{nx^n}]\\
&\sum_{d\in\mbb{Z}}\frac{P_{d,-1}}{q^{d\delta_{d<0}}x^d}=\frac{q}{u}\cdot\exp[-\sum_{n=1}^{\infty}\frac{P_{-n,0}}{nx^{-n}}]\exp[-\sum_{n=1}^{\infty}\frac{P_{n,0}}{nx^nq^n}]
\end{align}

The elements $P_{d,1}$ has the expression:
\begin{align}
P_{k,\pm1}=\text{Sym}[\frac{z_k^{\pm1}}{\prod_{i=1}^{k-1}(1-\frac{qz_{i+1}}{z_{i}})}\prod_{i<j}\zeta(\frac{z_i}{z_j})]
\end{align}

$P_{d,1}$ and $P_{-d,1}$ has the following commutation relations:
\begin{align}
[P_{d,1},P_{d',-1}]=\frac{(1-q_1)(1-q_2)}{(1-q^{-1})}(\delta_{d+d'\geq0}\frac{Q_{d+d',0}}{q^{-d'\delta_{d'<0}}}-\delta_{d+d'\leq 0}\frac{Q_{d+d',0}}{q^{d'\delta_{d'>0}}})
\end{align}
where 
\begin{align}
\sum_{n=0}^{\infty}\frac{Q_{\pm n,0}}{x^{\pm n}}=\exp[\sum_{n=1}^{\infty}\frac{P_{\pm n,0}}{nx^{\pm n}}(1-q^{-n})]
\end{align}

If we choose $d,d'\geq0$, we have that:
\begin{align}
[P_{d,1},P_{d',-1}]=\frac{(1-q_1)(1-q_2)}{(1-q^{-1})}Q_{d+d',0}
\end{align}

If we write the expression of $Q_{d+d',0}$ in terms of the shuffle elements, we have that:
\begin{align}
Q_{d+d',0}=\frac{(1-q^{-1})}{(1-q_1)(1-q_2)}\text{Sym}[\frac{u_{d,d'}(z_1,\cdots,z_{d+d'})}{\prod_{i=1}^{d+d'}(1-\frac{qz_{i+1}}{z_{i}})}\prod_{i<j}\zeta(\frac{z_{i}}{z_{j}})]
\end{align}

Here
\begin{align}
u_{d,d'}=\sum_{\sigma\in S_{d},\sigma'\in S_{d'}}(\frac{z_{\sigma(d)}}{z_{\sigma'(d+d')}}-\frac{z_{\sigma(d+d')}}{z_{\sigma'(d')}})(1-\frac{qz_{\sigma'(d+1)}}{z_{\sigma(d)}})
\end{align}
So iteratively one could obtain the expression of $P_{\pm n,0}$ in terms of $Q_{\pm n,0}$ via the following formal relations:
\begin{equation}
\begin{aligned}
\sum_{n=1}^{\infty}\frac{P_{\pm n,0}}{nx^{\pm n}}(1-q^{-n})=&\log(\sum_{n=0}^{\infty}\frac{Q_{\pm n,0}}{x^{\pm n}})\\
=&\sum_{k=1}^{\infty}\frac{1}{k}(\sum_{n=0}^{\infty}\frac{Q_{\pm n,0}}{x^{\pm n}}-1)^k
\end{aligned}
\end{equation}

\end{proof}

In the last section, we have seen that the matrix coefficients of the generators $P_{k,0}$ can be written as:

\begin{equation}
\begin{aligned}
\langle\lambda|P_{k,0}|\mu\rangle=P_{k,0}(\lambda\backslash\mu)(\frac{(1-q_1)(1-q_2)}{1-q})^{|\lambda\backslash\mu|}\prod_{\blacksquare\in\lambda\backslash\mu}\frac{\prod_{\square\text{ o.c. of }\lambda}[\frac{\chi_{\square}}{\chi_{\blacksquare}}]}{\prod_{\square\text{ i.c. of }\lambda}[\frac{\chi_{\square}}{\chi_{\blacksquare}}]}\in\mbb{Q}(q_1,q_2)
\end{aligned}
\end{equation}

Using the above type of formula, we define the degenerate limit of the matrix coefficients.
\begin{defn}
Given a rational function $f\in\mbb{Q}(q_1,q_2)$, using the substitution $ch(q_1)=e^{\hbar t_1}$ and $ch(q_2)=e^{\hbar t_2}$. We define the \textbf{asymptotic expansion }of $ch(f)$ as the Laurent expansion of $ch(f)$ around $\hbar=0$ in $\mbb{Q}[[\hbar^{\pm1}]]((t_1,t_2))$.
\end{defn}

The following lemma is easy to prove:
\begin{lem}
The asymptotic expansion of a rational function $f\in\mbb{Q}(q_1,q_2)$ lives in the space $\hbar^{-n}\mbb{Q}[[\hbar]]((t_1,t_2)$ for some $n\in\mbb{Z}$.
\end{lem}

We fix a minimal integer number $n\in\mbb{Z}$ such that for the corresponding $f\in\mbb{Q}(q_1,q_2)$ with the asymptotic expansion $ch(f)$ such that $\hbar^nch(f)\in\mbb{Q}[[\hbar]]((t_1,t_2))$.

\begin{defn}
Fix a rational function $f\in\mbb{Q}(q_1,q_2)$, given its asymptotic expansion $ch(f)$ and the minimal integer $n\in\mbb{Z}$ such that $\hbar^nch(f)\in\mbb{Q}[[\hbar]]((t_1,t_2))$. We define its \textbf{degenerate limit } $f^a$ as the following:
\begin{align}
f^a:=\hbar^nch(f)\text{mod }\hbar\in\mbb{Q}((t_1,t_2))
\end{align}
\end{defn}

Actually it is easy to see that given $f\in\mbb{Q}(q_1,q_2)$, the degenerate limit $f^a$ is a rational function over $t_1,t_2$, i.e. $f^a\in\mbb{Q}(t_1,t_2)$.

Using the degenerate limit, we can express the degenerate limit of $\langle\lambda|P_{ak,0}|\mu\rangle$ as :
\begin{equation}
\begin{aligned}
\langle\lambda|P_{ak,0}^a|\mu\rangle=&\text{Sym}[\frac{\omega^a(ak)}{\prod_{i=1}^{ak-1}(t_1+t_2+\chi_{i+1}-\chi_{i})}\prod_{i<j}\zeta_{a}(\chi_{i}-\chi_{j})]\prod_{\blacksquare\in\bm{\lambda}\backslash\bm{\mu}}\frac{\prod_{\square\text{ o.c. of }\lambda}(\chi_{\square}-\chi_{\blacksquare})}{\prod_{\square\text{ i.c. of }\lambda}(\chi_{\square}-\chi_{\blacksquare})}\\
=&a^3k^3t_1t_2\langle J_{\lambda},\frac{\partial}{\partial p_{-ak}}J_{\mu}\rangle
\end{aligned}
\end{equation}

In the last step we use the fact that $P_{-n,0}=-n(1-q^{n}_1)(1-q_2^n)\frac{\partial}{\partial p_{n}}$ and the Macdonald polynomial $P_{\lambda}$ degenerate to the Jack polynomial $J_{\lambda}$ as $\hbar\rightarrow0$. Briefly speaking, the degenerate limit $P^{a}_{ak,0}$ can be written as:
\begin{align}
P^{a}_{ak,0}=a^2k^2t_1t_2\alpha_{ak}
\end{align}

Here $\omega^a(ak)\in\mbb{Q}$ is some rational number corresponding to the Laurent polynomial $\omega(z_1,\cdots,z_{ak})$.

Using these formulas for the generators, in this section we prove the following degeneration theorem:
\begin{thm}\label{degeneration-thm-instanton}
The degeneration limit of the algebraic quantum difference equation for the instanton moduli space $M(n,r)$ over $K_{T}(M(n,r))$ coincides with the Dubrovin connection over $H_{T}^*(M(n,r))$.
\end{thm}
\begin{proof}

From the previous discussion, the proof can be reduced to compute the matrix coefficients $\langle\bm{\lambda}|R|\bm{\mu}\rangle$.

Also note that the action of $U_{q,t}(\hat{\hat{\mf{gl}}}_1)$ over $K_{T}(M(r))$ can be deduced via the isomorphism $K_{T}(M(r))\cong K_{T}(M(1))^{\otimes r}$. Thus the proof is reduced to $M(1)$, i.e. the Hilbert scheme of points $\text{Hilb}_{n}(\mbb{C}^2)$.

Now we take value in the partition $\bm{\lambda}\backslash\bm{\mu}$ and take the limit $q=e^{\hbar(t_1+t_2)}$ as $\hbar\rightarrow0$, in this case we choose suitable normalization factor $U_{ak}(q,t)$ such that as $\hbar\rightarrow0$, $U_{ak}(q,t)$:
\begin{equation}
\begin{aligned}
&\langle\bm{\lambda}|P^{a}_{ak,bk}|\bm{\mu}\rangle:=\lim_{\hbar\rightarrow0}\frac{\langle\bm{\lambda}|P_{ak,bk}(\hbar)|\bm{\mu}\rangle-1}{\hbar^{\text{leading order}}}\\
=&\hbar^{1+ak(\#\Delta(\lambda))-\text{leading order}}\text{Sym}[\frac{k}{\prod_{i=1}^{ak-1}(t_1+t_2+\chi_{i+1}-\chi_{i})}\prod_{i<j}\zeta_{a}(\chi_{i}-\chi_{j})]\prod_{\blacksquare\in\bm{\lambda}\backslash\bm{\mu}}\frac{\prod_{\square\text{ o.c. of }\lambda}(\chi_{\square}-\chi_{\blacksquare})}{\prod_{\square\text{ i.c. of }\lambda}(\chi_{\square}-\chi_{\blacksquare})}\\
=&\text{Sym}[\frac{k}{\prod_{i=1}^{ak-1}(t_1+t_2+\chi_{i+1}-\chi_{i})}\prod_{i<j}\zeta_{a}(\chi_{i}-\chi_{j})]\prod_{\blacksquare\in\bm{\lambda}\backslash\bm{\mu}}\frac{\prod_{\square\text{ o.c. of }\lambda}(\chi_{\square}-\chi_{\blacksquare})}{\prod_{\square\text{ i.c. of }\lambda}(\chi_{\square}-\chi_{\blacksquare})}
\end{aligned}
\end{equation}

i.e.
\begin{align}\label{degeneration-of-generators}
P^{a}_{\pm ak,\pm bk}=\frac{\omega_a(\pm ak)}{k}P^{a}_{\pm ak,0}
\end{align}

Also note that:
\begin{equation}
\begin{aligned}
&\langle\bm{\lambda}|P^{a}_{-ak,0}P^{a}_{ak,0}|\bm{\mu}\rangle:=\lim_{\hbar\rightarrow0}\frac{\langle\bm{\lambda}|P_{-ak,0}P_{ak,0}(\hbar)|\bm{\mu}\rangle-1}{\hbar}\\
=&-\hbar^2a^3k^3\langle\bm{\lambda}|t_1t_2p_{-ak}\frac{\partial}{\partial p_{-ak}}|\bm{\mu}\rangle=\hbar^2a^2k^2\langle\bm{\lambda}|t_1t_2\alpha_{-ak}\alpha_{ak}|\bm{\mu}\rangle
\end{aligned}
\end{equation}

This means that $\omega_{a}(-ak)\omega_{a}(ak)=a^2k^2$. So in this way we find that:
\begin{align}
P^{a}_{-ak,-bk}P_{ak,bk}^a=a^{-2}P^{a}_{-ak,0}P^{a}_{ak,0}=k^2\alpha_{-ak}\alpha_{ak}
\end{align}

Thus we can compute that:
\begin{equation}
\begin{aligned}
&\langle\bm{\mu}_2|n_{k}^a\alpha^{a}_{-ak,-bk}\alpha^{a}_{ak,bk}|\bm{\mu}_1\rangle:=\lim_{\hbar\rightarrow0}\frac{\langle\bm{\mu}_2|n_{k}\alpha_{-ak,-bk}\alpha_{ak,bk}|\bm{\mu}_1\rangle-1}{\hbar}\\
=&\lim_{\hbar\rightarrow0}\frac{1}{\hbar}(\sum_{\bm{\lambda}}\frac{(q_1^{k/2}-q_1^{-k/2})(q_2^{k/2}-q_2^{-k/2})(q^{k/2}-q^{-k/2})}{k(q_1^k-1)^2(1-q_2^k)^2}\langle\mu_2|P_{-ak,-bk}|\lambda\rangle\langle\lambda|P_{ak,bk}|\mu_1\rangle-1)\\
=&\lim_{\hbar\rightarrow0}\frac{1}{\hbar}\sum_{\lambda}\frac{k^3t_1t_2(t_1+t_2)}{k^5t_1^2t_2^2\hbar}\hbar^{2+ak(\#\Delta(\lambda)-\#\Delta(\mu_2)))}\langle\mu_2|P_{-ak,-bk}^a|\lambda\rangle\langle\lambda|P_{ak,bk}|\mu_1\rangle\\
=&\sum_{\lambda}\frac{(t_1+t_2)}{k^2t_1t_2}\langle\mu_2|P_{-ak,-bk}^a|\lambda\rangle\langle\lambda|P_{ak,bk}^a|\mu_1\rangle=\frac{(t_1+t_2)}{k^2t_1t_2}\langle\mu_2|P_{-ak,-bk}^aP_{ak,bk}^a|\mu_1\rangle\\
=&(t_1+t_2)\langle\mu_2|\alpha_{-ka}\alpha_{ka}|\mu_1\rangle
\end{aligned}
\end{equation}

Back to the quantum difference operator $\mbf{M}_{\mc{O}(1)}(z)$, one can find that the corresponding degeneration limit should be written as:
\begin{equation}
\begin{aligned}
&\mc{O}(1)_{cl}+(t_1+t_2)\sum_{a/b\in\text{Wall}}\sum_{k=1}^{\infty}\frac{1}{1-z^{-kb}}\alpha_{-kb}\alpha_{kb}\\
=&\mc{O}(1)_{cl}+(t_1+t_2)\sum_{k=1}^{\infty}\frac{k}{1-z^{-k}}\alpha_{-k}\alpha_{k}
\end{aligned}
\end{equation}

Thus we have finished the proof.
\end{proof}

\subsection{Monodromy and connection matrix from the degeneration}
We first write down the fundamental solution of the algebraic quantum difference equation around $z=0$ and $z=\infty$. For the detailed analysis of the solution one can read \cite{S21}.

We know that $\mbf{M}_{\mc{O}(1)}(0)=\mc{O}(1)$ is diagonal in the basis of fixed point basis. Let $P$ be the matrix with columns given by eigenvectors $P_{\lambda}$. We denote $E_{0}$ the digaonal matrix of eigenvalues, and we have that:
\begin{align}
\mbf{M}_{\mc{O}(1)}(0)P=PE_0
\end{align}
There exists a unique fundamental solution at $z=0$ written as:
\begin{align}
\Psi_{0}(z)=P\Psi^{reg}_{0}(z)\exp(\frac{\ln(E_0)\ln(z)}{\ln(q)}),\qquad\Psi^{reg}_{0}(z)=1+\sum_{k=1}^{\infty}\Psi_{0,k}^{reg}z^k
\end{align}
And $\Psi^{reg}_{0}(z)$ satisfies the following difference equation:
\begin{align}
\Psi^{reg}_{0}(pz)E_0=\mbf{M}_{\mc{O}(1)}^*(z)\Psi^{reg}_{0}(z),\qquad\mbf{M}_{\mc{O}(1)}^*(z)=P^{-1}\mbf{M}_{\mc{O}(1)}(z)P
\end{align}

Similarly for the fundamental solution near $z=\infty$, $\mbf{M}_{\mc{O}(1)}(\infty)$ is written as
\begin{align}
\mbf{M}_{\mc{O}(1)}(\infty)=\mc{O}(1)\prod^{\rightarrow}_{m\in\text{Walls}\cap[-1,0)}\mbf{m}((1\otimes S_{m})(R_{m}^{-})^{-1})
\end{align}
It can be proved that $\mbf{M}_{\mc{O}(1)}(\infty)$ is diagonlisable over $\mbb{Q}(q_1,q_2)$ with eigenvalues given by rational functions in $q_1,q_2$. Now let $H^*$ be the matrix with columns given by the eigenvectors of $\mbf{M}_{\mc{O}(1)}(\infty)$. Let $E_{\infty}$ be the diagonal matrix of eigenvalues so that:
\begin{align}\label{opposite-eigenvectors}
\mbf{M}_{\mc{O}(1)}(\infty)H^*=H^*E_{\infty}
\end{align}
The unique fundamental solution at $z=\infty$ can be written as:
\begin{align}
\Psi_{\infty}(z)=H^*\Psi^{reg}_{\infty}(z)\exp(\frac{\ln(E_{\infty})\ln(z)}{\ln(q)}),\qquad\Psi^{reg}_{\infty}(z)=1+\sum_{k=1}^{\infty}\Psi^{reg}_{\infty,k}z^{-k}
\end{align}
And $\Psi^{reg}_{\infty}$ solves the equation
\begin{align}
\Psi^{reg}_{\infty}(pz)E_{\infty}=\mbf{M}_{\mc{O}(1)}^*(z)\Psi^{reg}_{\infty}(z),\qquad \mbf{M}_{\mc{O}(1)}^*(z)=(H^*)^{-1}\mbf{M}_{\mc{O}(1)}(z)H^*
\end{align}

The following identity is really useful and was proved in \cite{S21}:
\begin{lem}
Let $T=P^{-1}H^*$. Then we have the following identity:
\begin{align}
(\prod^{\rightarrow}_{w\in[0,1)}\mbf{B}_{w})E_0=TE_{\infty}T^{-1}
\end{align}
\end{lem}
\begin{proof}
Using the formula \ref{opposite-eigenvectors} and the definition of $T$.
\end{proof}

Now we define the connection matrix as:
\begin{align}\label{connection-matrix}
\textbf{Mon}(z):=\Psi_{0}(z)^{-1}\Psi_{\infty}(z)
\end{align}

We also define the regular part of the connection matrix by:
\begin{align}\label{regular-connection-matrix}
\textbf{Mon}^{reg}(z):=\Psi_{0}^{reg}(z)^{-1}\Psi_{\infty}^{reg}(z)
\end{align}

It is obvious that $\textbf{Mon}(z)$ is $p$-period, i.e. $\textbf{Mon}(pz)=\textbf{Mon}(z)$, and by construction it is written as the following:
\begin{align}
\textbf{Mon}(z,q_1,q_2,p)=\exp(-\frac{\ln(\mbf{E}_0)\ln(z)}{\ln(p)})\textbf{Mon}^{reg}(z,q_1,q_2,p)\exp(\frac{\ln(\mbf{E}_{\infty})\ln(z)}{\ln(p)})
\end{align}
such that $\textbf{Mon}^{reg}(z,q_1,q_2)$ can be written as the combination of the theta-functions of the modular parametre $p$. 

One aspect is on the connection matrix. The connection matrix of the quantum difference equation over the instanton moduli space $M(n,r)$ can also be used to compute the connection matrix of the corresponding quantum differential equation. 

Using the connection matrix of the quantum difference equation, we define the connection matrix for the Dubrovin connection by:
\begin{align}\label{defn-of-trans-matrix}
\text{Trans}(s)=\lim_{\tau\rightarrow0}\textbf{Mon}(e^{2\pi is},e^{2\pi t_1\tau},e^{2\pi t_2\tau},e^{2\pi i\tau})
\end{align}

The following theorem can also be derived from \cite{S21}\cite{Z24}:
\begin{prop}\label{transport-matrix}
The connection matrix $\text{Trans}(s)$ of the Dubrovin connection from $z=0$ to $z=\infty$ along a line $\gamma_s$ intersecting $|z|=1$ at a non-singular point $z=e^{2\pi is}$ equals:
\begin{align}
\text{Trans}(s)=
\begin{cases}
\prod_{w\in(0,s)}^{\leftarrow}(\mbf{B}_{w}^*)^{-1}\cdot T,\qquad s\geq0\\
\prod_{w\in(s,0)}^{\rightarrow}\mbf{B}_{w}^*\cdot T,\qquad s<0
\end{cases}
\end{align}
\end{prop}

Using the connection matrix, one can obtain the expression for the monodromy representations via $\text{Trans}(s')^{-1}\text{Trans}(s)$. The Dubrovin connection over the instanton moduli space $M(n,r)$ has regular singularities on the following isolated points:
\begin{align}
\text{Sing}:=\{0,\infty,\exp(\frac{2\pi ia}{b})\}_{1\leq a\leq b\leq n)}
\end{align}

Using the deduction as in \cite{S21} and \cite{Z24}, we have the following description for the monodromy representation:
\begin{thm}\label{monodromy-representation}
For the monodromy representation of the Dubrovin connection of the instanton moduli space:
\begin{equation}
\begin{aligned}
&\Phi:\pi_1(\mbb{P}^1\backslash\text{Sing},0^{+})\rightarrow\text{Aut}(H^*_{T}(M(n,r)))\\
&\gamma_{w}\mapsto\mbf{B}_{w}^*,\qquad\gamma_0\mapsto E_{0},\qquad\gamma_{\infty}\mapsto TE_{\infty}T^{-1}
\end{aligned}
\end{equation}
The monodromy group is generated by $\mbf{m}((1\otimes S_{\frac{a}{b}})(R_{\frac{a}{b}}^{-})^{-1})$. Here $R_{\frac{a}{b}}^{-}:=(R_{\frac{a}{b}})_{21}$ is the universal $R$-matrix for the slope subalgebra $\mc{B}_{\frac{a}{b}}\cong U_{q}(\hat{\mf{gl}}_1)$ with $q=\exp(2\pi i(t_1+t_2))$ with respect to the opposite coproduct $\Delta_{\frac{a}{b}}^{op}$.
\end{thm}
\begin{proof}
The proof is more or less similar as in \cite{S21}. For the fundamental solution $\psi_{0}(z)$ around $z=0$, fix a closed curve $\gamma_0$ around $z=0$. Going around $\gamma_0$ gives rise to the following transformation:
\begin{align}
\psi_0(\gamma_0\cdot z)=\psi_0(z)E_0
\end{align}
Then we have that $\gamma_0\mapsto E_0$.

Similarly for the fundamental solution $\psi_{\infty}(z)$ around $z=\infty$, choosing a loop $\gamma_{\infty}$ around $z=\infty$ we have that:
\begin{align}
\psi_{\infty}(\gamma_{\infty}\cdot z)=\psi_{\infty}(z)E_{\infty}
\end{align}
Thus we have that $\gamma_{\infty}\mapsto TE_{\infty}T^{-1}$ since $\text{Trans}(0)=T$.

For the loop $\gamma_{w}$ around other singuarlities, consider $s',s\in\mbb{R}$ such that there is only one singuarlity $w$ such that $s'>w>s$. Then the composition $\text{Trans}(s)\text{Trans}(s')^{-1}$ means that the solution $\psi_{0}(z)$ is analytically continued to the loop around the singuarlity $w$ by the loop $\gamma_{w}$. In this way by the proposition \ref{transport-matrix}, $\text{Trans}(s)\text{Trans}(s')^{-1}=\mbf{B}^{*}_{w}$. Thus in this way we finish the proof.

\end{proof}

In terms of the generators, we have that:
\begin{align}\label{mono-element}
\mbf{m}((1\otimes S_{\frac{a}{b}})(R_{\frac{a}{b}}^{-})^{-1})=\exp(-\sum_{k=1}^{\infty}n_{k}P_{bk,ak}P_{-bk,-ak})
\end{align}

This implies that we can use the fixed point basis matrix to express the monodromy operators explicitly.

\section{\textbf{Degeneration limit for the Okounkov-Smirnov quantum difference equation}}

In this section we prove that the Okounkov-Smirnov geometric quantum difference equation also degenerates to the Dubrovin connection for the instanton moduli space.

\subsection{Degeneration limit of the Okounkov-Smirnov quantum difference equation}
Similar to the case of the algebraic quantum difference equation, the geometric quantum difference equation also has the degeneration limit as the Dubrovin connection:

\begin{thm}\label{main-theorem-1}
The degeneration limit of the Okounkov-Smirnov quantum difference equation is equivalent to the Dubrovin connection for the instanton moduli space.
\end{thm}

\begin{proof}
The strategy of the proof is similar to the one in \cite{Z24-2}. First we can prove that $\mf{n}_{w}^{MO}$ is the classical limit of the wall subalgebra $U_{q}^{MO}(\mf{n}_{w})$:
\begin{lem}\label{q-1-degeneration}
There is an isomorphism:
\begin{align}
U_{q}^{MO}(\mf{n}_w)/(q-1)U_{q}^{MO}(\mf{n}_w)=U(\mf{n}_w^{MO})
\end{align}
\end{lem}
\begin{proof}
By definition $U_{q}^{MO}(\mf{n}_{w})$ is generated by the matrix coefficients of $R_{w}^+=\text{Id}+U_{w}^+$. The left-ideal $(q-1)U_{q}^{MO}(\mf{n}_{w})$ is generated by the matrix coefficients of $(q-1)R_{w}^+$, and now by definition:
\begin{align}
\lim_{q\rightarrow1}\frac{R_w^+-\text{Id}}{q-1}=r_{w}^+
\end{align}

which is the generators for the Lie algebra $\mf{n}_w^{MO}$.
\end{proof}

It is easy to see that $U_{q}^{MO}(\mf{n}_w)$ and $U(\mf{n}_w^{MO})$ has the grading:
\begin{align}
U_{q}^{MO}(\mf{n}_w)=\bigoplus_{k\in\mbb{N}}U_{q}^{MO}(\mf{n}_w)_{k},\qquad U(\mf{n}_w^{MO})=\bigoplus_{k\in\mbb{N}}U(\mf{n}_w^{MO})_{k}
\end{align}

From the Lemma \ref{q-1-degeneration} we can have that:
\begin{align}
\text{dim}(U(\mf{n}_{w}^{MO})_{k})\leq\text{dim}(U_{q}^{MO}(\mf{n}_w)_{k})
\end{align}

Moreover, for the Lie algebra $\mf{g}_{Q}$, we have the decomposition:
\begin{align}
\mf{g}_{Q}^{MO}=\bigcup_{w\in\mbb{Q}\cap[0,1)}\mf{g}_w^{MO}
\end{align}

The relation between $\mf{g}^{MO}_{Q}$ and $\hat{\mf{gl}}_1$ has been proved in \cite{MO12} and \cite{BD23}:
\begin{thm}
There is an isomorphism of Lie algebras:
\begin{align}
\mf{g}^{MO}_{Q}\cong\hat{\mf{gl}}_1
\end{align}
up to some centre.
\end{thm}

By construction, the Lie subalgebra $\mf{n}_{w}^{MO}$ with $w=a/b$ consists of the root vectors $e_k:H(n,r)\rightarrow H(n+kb,r)$.

Now note that $\text{dim}(\mf{n}_{w}^{MO})=\text{dim}(\mf{n}_{w})$ and since we have the embedding $\mf{n}_{w}\hookrightarrow\mf{n}_{w}^{MO}$, we have the isomorphism $\mf{n}_{kb}\cong\mf{n}_w\subset\mf{n}_{w}^{MO}\subset\mf{n}_{kb}^{MO}$. This implies that $\mf{n}^{MO}_{a/b}\cong\mf{n}_{a/b}$ consists of the root vectors of the form $\{\alpha_{kb}\}_{k\geq1}$. This also implies that the classical $r$-matrix $r_{a/b}^{\pm}$ can be written as:
\begin{align}
r_{a/b}^{\pm}=\sum_{k\geq1}\alpha_{\pm bk}\otimes\alpha_{\mp bk}
\end{align}

Now we proceed to prove the Theorem \ref{main-theorem-1}. The degeneration limit of the geometric quantum difference opearator \ref{geometric-doperator} can be written as:

\begin{equation}
\begin{aligned}
\mbf{M}^{s,coh}_{\mc{O}(1)}(\lambda)=&c_1(\mc{O}(1))+\sum_{w\in[s,s+\mc{L})}\mbf{B}_{w}^{coh}(\lambda)\\
=&c_1(\mc{O}(1))+\sum_{w\in[s,s+\mc{O}(1))}\frac{1}{1-Ad_{q_{(1)}^{\lambda}}}\mbf{m}(r_{w}^-)\\
=&c_1(\mc{O}(1))+\sum_{a/b\in[s,s+\mc{O}(1))}\sum_{k\geq1}\frac{1}{1-Ad_{q_{(1)}^{\lambda}}}\alpha_{bk}\alpha_{-bk}\\
=&c_1(\mc{O}(1))+\sum_{k\geq0}\frac{k}{1-z^{k}}\alpha_{k}\alpha_{-k}
\end{aligned}
\end{equation}

This coincides with the Dubrovin connection of the instanton moduli space. Thus this finishes the proof.

\end{proof}

Using the similar proof of the main theorem in \cite{Z24-2}\cite{Z24}, we can conclude that:
\begin{thm}\label{geometric-monodromy}
For the monodromy representation of the Dubrovin connection:
\begin{align}
\pi_1(\mbb{P}^1\backslash\text{Sing},0^+)\rightarrow\text{Aut}(H_{T}^*(M(n,r)))
\end{align}
It is generated by $\mbf{m}((1\otimes S_{w})(R_{w}^{-,MO})^{-1})$. Here $R_{w}^{-,MO}$ is the universal $R$-matrix for the wall subalgebra $U_{q}^{MO}(\mf{g}_{w})$.
\end{thm}

Using the result in \cite{Z24-2}, we can conclude that:
\begin{prop}
\begin{align}
\mbf{m}((1\otimes S_{a/b})(R_{a/b}^{-,MO})^{-1})=\mbf{m}((1\otimes S_{a/b})(R_{a/b}^{-})^{-1})
\end{align}
\end{prop}

\section{\textbf{Isomorphism of MO quantum affine algebra and quantum toroidal algebra}}

In this section we prove that the positive half and negative half of the quantum toroidal algebra $U_{q,t}(\hat{\hat{\mf{gl}}}_1)$ is isomorphic to the Maulik-Okounkov quantum affine algebra:

\subsection{Stable basis and the Pieri rule}
A good set of basis for $K(r):=\bigoplus_{n}K_{T}(M(n,r))$ under the action of the slope subalgebra $\mc{B}_{m}$ of slope $m\in\mbb{Q}$ is the stable basis $s^{m}_{\bm{\lambda}}$ of the slope $m\in\mbb{Q}$.

The stable basis is constructed via the following way: We choose the stable envelope as:
\begin{equation}
\text{Stab}_{\sigma,m}:K_{T}(M(n,r)^T)\rightarrow K_{T}(M(n,r))
\end{equation}
Here $T=A\times\mbb{C}^*_{q_1}\times\mbb{C}^*_{q_2}$ is the product of the maximal torus and the two-dimensional torus $\mbb{C}^*_{q_1}\times\mbb{C}^*_{q_2}$. It is known that the fixed points $M(n,r)^T$ of $T$-torus action are isolated, which is labeled by the multi-partitions $\bm{\lambda}$. Now we define $t:=(\frac{q_1}{q_2})^{1/2}$, and we will fix on the $t$-degree of the $K$-theory class. Fixing the character $\sigma:\mbb{C}^*\rightarrow T$ that corresponds to the isolated fixed points subset. We denote the image of $\sigma$ by $u_i\mapsto t^{N_i}$ such that $N_1<<N_2<<\cdots<<N_{r}$. In this case the fixed point basis are ordered by the dominance partial order of the partitions $\bm{\lambda}$. We would always denote the stable basis as $s^{m}_{\bm{\lambda}}:=\text{Stab}_{\sigma,m}(|\bm{\lambda}\rangle)$.

The diagonal part of $s^{m}_{\bm{\lambda}}$ is given by:
\begin{align}
s^{m}_{\bm{\lambda}}|_{\bm{\lambda}}=k_{\bm{\lambda}}=[T^{-}_{\bm{\lambda}}M(n,r)]=\prod_{\square\in\bm{\lambda}}[\prod_{\blacksquare\in\bm{\lambda}}\zeta(\frac{\chi_{\blacksquare}}{\chi_{\square}})\prod_{k=1}^{r}[\frac{u_k}{q\chi_{\square}}][\frac{\chi_{\square}}{qu_k}]]^{(-)}
\end{align}
Here $(-)$ stands for the product of $[x]$ such that the $t$-degree of $[x]$ is smaller than $0$. The notation for $(+)$ and $(0)$ has the similar meaning.

It is known that the slope subalgebra $\mc{B}_{m}$ is generated by $P^{m}_{k}$ or $E^{m}_{k}$ in \cite{N14}. 

Now given an operator $F:K(n',r)\rightarrow K(n,r)$ whose matrix coefficients in terms of fixed point basis is given by:
\begin{align}
F|\bm{\mu}\rangle=\sum_{|\bm{\lambda}|=n}F^{\bm{\lambda}}_{\bm{\mu}}\cdot|\bm{\lambda}\rangle
\end{align}

If we choose the stable basis $s^{m}_{\bm{\lambda}}$ of the slope $m$, we write down the coefficients with respect to the stable basis as:
\begin{align}
F\cdot s^{m}_{\bm{\mu}}=\sum_{|\bm{\lambda}|=n}\gamma^{\bm{\lambda}}_{\bm{\mu}}\cdot s^{m}_{\bm{\lambda}}
\end{align}

The good property about the stable basis $s^{m}_{\bm{\mu}}$ is that it subtracts the special information of the action of the slope subalgebra $\mc{B}_{m}$.

\begin{lem}{\cite{N15}}\label{Pieri-rule}
For $F$ such that:
\begin{equation}
\begin{aligned}
&\text{max deg }F^{\bm{\lambda}}_{\bm{\mu}}\leq k_{\bm{\lambda}}-k_{\bm{\mu}}+m(c_{\bm{\lambda}}-c_{\bm{\mu}})\\
&\text{min deg}F^{\bm{\lambda}}_{\bm{\mu}}\geq k_{\bm{\mu}}-k_{\bm{\lambda}}+m(c_{\bm{\lambda}}-c_{\bm{\mu}})
\end{aligned}
\end{equation}
 we have that:
\begin{align}\label{Lowest-stable-formula}
\gamma^{\bm{\lambda}}_{\bm{\mu}}=(\text{l.d.}F^{\bm{\lambda}}_{\bm{\mu}})\cdot\frac{\text{l.d.}(k_{\bm{\mu}})}{\text{l.d.}(k_{\bm{\lambda}})}
\end{align}

Here $\text{l.d.}$ stands for the lowest $t$-degree part with $t:=(\frac{q_1}{q_2})^{1/2}$. 
\end{lem}

\subsection{Pieri-rule Formula}

It is proved in \cite{N16} that the generators $E^{a/b}_{k}$ acts on the stable basis $s^{a/b}_{\lambda}$ of slope $a/b$ in the following way:
\begin{align}
E_{k}^{a/b}s^{a/b}_{\mu}=\sum_{\lambda}s^{a/b}_{\lambda}(-1)^{ht}\prod_{i=1}^k\prod_{j=1}^n\chi_{j}(B_i)^{\lfloor\frac{aj}{b}\rfloor-\lfloor\frac{a(j-1)}{b}\rfloor}
\end{align}
with the sum going over all vertical $k$-strips of $n$-ribbons of shape $\lambda\backslash\mu$. $\chi_{j}(B):=q_1^{x+y}q_2^{x-y}$ is the weight function of the $j$-th box on the $i$-th ribbon.

Since the formula in terms of the stable basis is on the lowest $t$-degree of $E^{a/b}_{k}(\bm{\lambda}\backslash\bm{\mu})$ by \ref{Lowest-stable-formula}. Similarly for the $r$-partitions $\bm{\lambda}$, we have the similar formula:
\begin{prop}{(Pieri rule formula)}
For the stable basis $s^{a/b}_{\bm{\mu}}$ of $K_{T}(M(n,r))$, we have that:
\begin{equation}
\begin{aligned}
\text{l.d. }E^{a/b}_{k}(\bm{\lambda}\backslash\bm{\mu})=q^{|\bm{\lambda}\backslash\bm{\mu}|}q_1^{-\#_{\bm{\lambda}\backslash\bm{\mu}}}\prod_{i=1}^{k}\prod_{j=1}^n\chi_{j}(B_i)^{r_{\frac{a}{b}}(j)}\cdot\prod_{l=1}^r\frac{\prod^{\blacksquare\in\lambda_l\backslash\mu_l, c_{\blacksquare}=c_{\square}}_{\square\text{ outer corner of }\lambda\backslash\mu}[\frac{q\chi_{\blacksquare}}{\chi_{\square}}]}{\prod^{\blacksquare\in\lambda_l\backslash\mu_l,c_{\blacksquare}=c_{\square}}_{\square\text{ inner corner of }\lambda\backslash\mu}[\frac{q\chi_{\blacksquare}}{\chi_{\square}}]}
\end{aligned}
\end{equation}

Here $\chi_{j}^{(l_j)}(B_i)=u_{l_j}q_1^{x+y}q_2^{x-y}$, $u_{l_j}$ with $1\leq l_j\leq r$ is the equivariant parametres.
\end{prop}

Here $U_{q}^{MO}(\mf{g}_w)\subset U_{q}^{MO}(\mf{g})$ is the wall subalgebra for the wall $R$-matrices $R_{w}^{\pm}$ as defined before. The comparison between the slope subalgebra $\mc{B}_{\frac{a}{b}}$ and the corresponding wall subalgebra $U_{q}^{MO}(\mf{g}_{\frac{a}{b}})$ lies in the following proposition:
\begin{thm}
There exists a Hopf algebra embedding
\begin{align}
\mc{B}_{\frac{a}{b}}\hookrightarrow U_{q}^{MO}(\mf{g}_{\frac{a}{b}})
\end{align}
\end{thm}

\begin{proof}

The proof is almost the same as in the proof of the Proposition $4.2$ in \cite{Z24-2}. The key is to prove the following identity:
\begin{align}\label{Coproduct-matching}
\text{Stab}_{a/b}(\Delta_{\frac{a}{b}}(E^{a/b}_{k})(s^{a/b}_{\lambda_1}\otimes s^{a/b}_{\lambda_2}))=E^{a/b}_{k}s^{a/b}_{(\lambda_1,\lambda_2)}
\end{align}

The right hand side of \ref{Coproduct-matching} can be simplified via the formula \ref{Lowest-stable-formula}. We first compute the lowest $t$-degree of the matrix coefficients in terms of the fixed point basis using the formula \ref{fixed-point-formula}:

\begin{equation}
\begin{aligned}
&\text{l.d. }E^{a/b}_{k}(\bm{\lambda}\backslash\bm{\mu})\prod_{\blacksquare\in\bm{\lambda}\backslash\bm{\mu}}\text{l.d. }[\frac{(1-q_1)(1-q_2)}{1-q}\zeta(\frac{\chi_{\blacksquare}}{\chi_{\bm{\mu}}})\tau(q\chi_{\blacksquare})]\\
=&(\cdots)\text{l.d. }E^{a/b}_{k}(\bm{\lambda}\backslash\bm{\mu})\text{l.d. }\prod_{\blacksquare\in\bm{\lambda}\backslash\bm{\mu}}[\prod_{\square\in\bm{\mu}}\frac{[\frac{q_1\chi_{\blacksquare}}{\chi_{\square}}][\frac{q_2\chi_{\blacksquare}}{\chi_{\square}}]}{[\frac{\chi_{\blacksquare}}{\chi_{\square}}][\frac{q\chi_{\blacksquare}}{\chi_{\square}}]}\prod_{i=1}^r[\frac{q\chi_{\blacksquare}}{u_i}]\\
=&(\cdots)\text{l.d. }E^{a/b}_{k}(\bm{\lambda}\backslash\bm{\mu})\text{l.d. }\prod_{\blacksquare\in\bm{\lambda}\backslash\bm{\mu}}^{c_{\blacksquare}\neq c_{\square}}[\prod_{\square\in\bm{\mu}}\frac{[\frac{q_1\chi_{\blacksquare}}{\chi_{\square}}][\frac{q_2\chi_{\blacksquare}}{\chi_{\square}}]}{[\frac{\chi_{\blacksquare}}{\chi_{\square}}][\frac{q\chi_{\blacksquare}}{\chi_{\square}}]}]\prod_{\blacksquare\in\bm{\lambda}\backslash\bm{\mu}}\prod_{i\neq c_{\blacksquare}}^r[\frac{q\chi_{\blacksquare}}{u_i}]]\\
&\times\text{l.d. }\prod_{\blacksquare\in\bm{\lambda}\backslash\bm{\mu}}^{c_{\blacksquare}=c_{\square}}[\prod_{\square\in\bm{\mu}}\frac{(1-\frac{q_1\chi_{\blacksquare}}{\chi_{\square}})(1-\frac{q_2\chi_{\blacksquare}}{\chi_{\square}})}{(1-\frac{\chi_{\blacksquare}}{\chi_{\square}})(1-\frac{q\chi_{\blacksquare}}{\chi_{\square}})}]\prod_{\blacksquare\in\bm{\lambda}\backslash\bm{\mu}}\prod_{i=c_{\blacksquare}}[\frac{q\chi_{\blacksquare}}{u_i}]\\
=&(1-q_2)^{|\bm{\lambda}\backslash\bm{\mu}|}\text{l.d. }E^{a/b}_{k}(\bm{\lambda}\backslash\bm{\mu})\text{l.d.}\prod_{l=1}^{r}\prod_{\blacksquare\in\lambda_l\backslash\mu_l}\frac{\prod_{\square\text{ i.c. of }\mu_l}[\frac{q\chi_{\square}}{\chi_{\blacksquare}}]}{\prod_{\square\text{ o.c. of }\mu_l}[\frac{q\chi_{\square}}{\chi_{\blacksquare}}]}\prod_{\blacksquare\in\bm{\lambda}\backslash\bm{\mu}}\prod_{i=c_{\blacksquare+1}}^{r}(\sqrt{q_1q_2})^{a_{\blacksquare}+1}\prod_{i=1}^{c_{\blacksquare}-1}\sqrt{q_1q_2}^{-1-a_{\blacksquare}}\\
=&(1-q_2)^{|\bm{\lambda}\backslash\bm{\mu}|}q^{|\bm{\lambda}\backslash\bm{\mu}|}q_1^{-\#_{\bm{\lambda}\backslash\bm{\mu}}}\prod_{i=1}^{k}\prod_{j=1}^n\chi_{j}(B_i)^{r_{\frac{a}{b}}(j)}\prod_{\blacksquare\in\bm{\lambda}\backslash\bm{\mu}}\prod_{i=c_{\blacksquare+1}}^{r}\sqrt{q_1q_2}^{a_{\blacksquare}+1}\prod_{i=1}^{c_{\blacksquare}-1}\sqrt{q_1q_2}^{-1-a_{\blacksquare}}
\end{aligned}
\end{equation}

Here $(\cdots)$ stands for the term that does not affect the computations on both sides of \ref{Coproduct-matching}. The diagonal part $k_{\bm{\mu}}$ can be written as:
\begin{align}
k_{\bm{\mu}}=\prod_{\blacksquare\in\bm{\lambda}}[\prod_{\square\in\bm{\lambda}}\zeta(\frac{\chi_{\blacksquare}}{\chi_{\square}})\prod_{i=1}^{r}[\frac{q\chi_{\blacksquare}}{u_i}][\frac{u_i}{q\chi_{\blacksquare}}]]^{(-)}
\end{align}
Therefore:
\begin{equation}
\begin{aligned}
\text{l.d. }(k_{\bm{\mu}})=\text{l.d. }\prod_{\blacksquare,\square\in\bm{\lambda}}[\zeta(\frac{\chi_{\blacksquare}}{\chi_{\square}})]^{(-)}\text{l.d. }\prod_{\blacksquare\in\bm{\lambda}}\prod_{i=c_{\blacksquare}+1}^r[\frac{q\chi_{\blacksquare}}{u_i}]\prod_{i=1}^{c_{\blacksquare}-1}[\frac{u_i}{q\chi_{\blacksquare}}]
\end{aligned}
\end{equation}

Thus we have the formula:
\begin{equation}
\begin{aligned}
\frac{\text{l.d. }(k_{\bm{\mu}})}{\text{l.d. }(k_{\bm{\lambda}})}=(\cdots)\text{l.d. }\prod_{\blacksquare\in\bm{\lambda}\backslash\bm{\mu}}\prod_{i=c_{\blacksquare}+1}^r[\frac{q\chi_{\blacksquare}}{u_i}]^{-1}\prod_{i=1}^{c_{\blacksquare}-1}[\frac{u_i}{q\chi_{\blacksquare}}]^{-1}
\end{aligned}
\end{equation}

Combining the two terms we have:
\begin{equation}
\begin{aligned}
&(\text{l.d.}(E^{a/b}_{k})^{\bm{\lambda}}_{\bm{\mu}})\cdot\frac{\text{l.d.}(k_{\bm{\mu}})}{\text{l.d.}(k_{\bm{\lambda}})}=(\cdots)\text{l.d. }\prod_{\blacksquare\in\bm{\lambda}\backslash\bm{\mu}}\prod_{i=1}^{c_{\blacksquare}-1}\frac{[\frac{q\chi_{\blacksquare}}{u_i}]}{[\frac{u_i}{q\chi_{\blacksquare}}]}\\
=&(\cdots)\text{l.d.}\prod_{\blacksquare\in\bm{\lambda}\backslash\bm{\mu}}\prod_{i=1}^{c_{\blacksquare}-1}q^{1-a_{\blacksquare}}q^{1+a_{\blacksquare}}=(\cdots)\text{l.d.}\prod_{\blacksquare\in\bm{\lambda}\backslash\bm{\mu}}\prod_{i=1}^{c_{\blacksquare}-1}\sqrt{q_1q_2}^2\\
=&(\text{l.d.}(E^{a/b}_{l})^{\bm{\lambda}_1}_{\bm{\mu}})\cdot\frac{\text{l.d.}(k_{\bm{\mu}_1})}{\text{l.d.}(k_{\bm{\lambda}_1})}(\text{l.d.}(E^{a/b}_{k-l})^{\bm{\lambda}_2}_{\bm{\mu}_2})\cdot\frac{\text{l.d.}(k_{\bm{\mu}_2})}{\text{l.d.}(k_{\bm{\lambda}_2})}q^{lr_1}
\end{aligned}
\end{equation}

This coincides with the left hand side of the formula \ref{Coproduct-matching} using the coproduct formula \ref{coproduct-group-elements}.

\end{proof}

Now we can proceed to the main theorem on the isomorphism positive half of the Maulik-Okounkov quantum affine algebras and the positive half of the quantum toroidal algebras.

\begin{thm}\label{Isomorphism-theorem}
Given $U_{q}^{+}(\hat{\mf{g}}_{Q})$ the positive half of the Maulik-Okounkov qunatum affine algebra of Jordan type. There is an isomorphism of algebras:
\begin{align}
U_{q_1,q_2}^{+}(\hat{\hat{\mf{gl}}}_1)\cong U_{q}^{+}(\hat{\mf{g}}_{Q})
\end{align}
\end{thm}

\begin{proof}
Since $U_{q_1,q_2}^{+}(\hat{\hat{\mf{gl}}}_1)$ and $U_{q}^{+}(\hat{\mf{g}}_{Q})$ are generated by the slope subalgebras and wall subalgebras, it is equivalent to prove that there is an isomorphism:
\begin{align}\label{slope-iso}
\mc{B}_{\frac{a}{b}}^+\cong U^+_{q}(\mf{g}_{\frac{a}{b}})
\end{align}
between the slope subalgebra of the slope $a/b$ and the wall subalgebra at the wall $a/b$.

By the monodromy representation of the Dubrovin connection, we have that:
\begin{align}
\mbf{m}((1\otimes S_{a/b})(R_{a/b}^{\pm,MO}))=\mbf{m}((1\otimes S_{a/b})(R_{a/b}^{\pm}))\in U_{q}^{+}(\mf{g}_{a/b})_{\mbf{0}}
\end{align}

Denote $U_{a/b}^{\pm}:=(R_{a/b}^{\pm})^{-1}R_{a/b}^{\pm,MO}\in U_{q}^{\pm}(\mf{g}_{a/b})\otimes U_{q}^{\mp}(\mf{g}_{a/b})$. Since $\mbf{m}((1\otimes S_{a/b})(U_{a/b}^{\pm}))=1$, by the Lemma $7.7$ in \cite{Z24-2} we have that $R_{a/b}^{\pm}=R_{a/b}^{\pm,MO}$. This implies the isomorphism \ref{slope-iso}.

Now by the factorisation property of the quantum toroidal algebra $U_{q,t}^{+}(\hat{\hat{\mf{gl}}}_1)\cong\bigotimes^{\rightarrow}_{\mu\in\mbb{Q}}\mc{B}_{\mu}^{+}$ and the factorisation property of the Maulik-Okounkov quantum affine algebra. We conclude the isomorphism.
\end{proof}

A direct consequence of the theorem is that the two quantum differece equations coincide:
\begin{cor}
The quantum difference equations constructed from quantum toroidal algebra $U_{q,t}(\hat{\hat{\mf{gl}}}_1)$ and the Maulik-Okounkov quantum affine algebra $U_{q}(\hat{\mf{g}}_{Q})$ of Jordan type coincide.
\end{cor}
\begin{proof}
This is equivalent to prove that the quantum difference operators $\mc{B}^{s}_{\mc{O}(1)}$ on both sides are the same. In fact, since the quantum difference operators $\mc{B}^s_{\mc{O}(1)}$ is written as the ordered product of the monodromy operators $\mbf{B}_{w}(z)$, and the monodromy operators come from the fusion operators $\mbf{J}_{w}^{\pm}(z)$, and the fusion operators are unique determined by the $R$-matrix $R_{w}^{\pm}$. While the $R$-matrix for each wall are the same, thus the quantum difference operators are the same.
\end{proof}

\subsection{Extension of the isomorphism to the whole algebra}

In this subsection we compare the isomorphism of $\mc{B}_{\infty}$ and $U_{q}(\mf{g}_{\infty})$, i.e. the wall subalgebra at infinity generated by the geometric $R$-matrix $\mc{R}^{\infty}(u)$ at infinite slope.

It is known that the geometric $R$-matrix $\mc{R}^{\infty}(u)$ is given as:
\begin{align}
\mc{R}^{\infty}(u)=\frac{\wedge^*N_{X^{\sigma}\subset X}^{+}}{\wedge^*N_{X^{\sigma}\subset X}^{-}}\in\text{End}(K_{T}(X^{\sigma}))
\end{align}

Without loss of generality, we take $X=M(n,2)$ and $X^{\sigma}=\sqcup_{n_1+n_2=n}M(n_1,1)\times M(n_2,1)$. In this case the $R$-matrix $\mc{R}^{\infty}(u)$ is given by:
\begin{equation}
\begin{aligned}
\mc{R}^{\infty}(u)=&\prod_{n_1+n_2=n}\frac{\wedge^*N^{+}_{M(n_1,1)\times M(n_2,1)\subset M(n,2)}}{\wedge^*N^{-}_{M(n_1,1)\times M(n_2,1)\subset M(n,2)}}\\
=&\frac{\wedge^*(\frac{q_1V'}{V''})\wedge^*(\frac{q_2V'}{V''})\wedge^*(\frac{q^{-1}V'}{V''})}{\wedge^*(\frac{q_1^{-1}V'}{V''})\wedge^*(\frac{q_2^{-1}V'}{V''})\wedge^*(\frac{qV'}{V''})}\frac{\wedge^*(\frac{V'}{W''})}{\wedge^*(\frac{q^{-1}V'}{W''})}\frac{\wedge^*(\frac{qW'}{V''})}{\wedge^*(\frac{W'}{V''})}
\end{aligned}
\end{equation}

Here $V',V''$ are the tautological vector bundle over $M(n_1,1)$ and $M(n_2,1)$. $W',W''=\mbb{C}$ are the trivial bundle over $M(n_1,1)$ and $M(n_2,1)$ corresponding to the framing vectors.

The slope subalgebra $\mc{B}_{\infty}$ of infinite slope is generated by $\{h_{0,\pm d}\}$ and the corresponding universal $R$-matrix is given by:
\begin{equation}
\begin{aligned}
R_{\infty}=\exp[-\sum_{d=1}^{\infty}\frac{(q_1^{d/2}-q_1^{-d/2})(q_2^{d/2}-q_2^{-d/2})(q^{d/2}-q^{-d/2})}{d}h_{d}\otimes h_{-d}]
\end{aligned}
\end{equation}

\begin{prop}
The universal $R$-matrix of infinite slope $R_{\infty}$ coincides with $\mc{R}^{\infty}(u)$ up to a function $N(u)$.
\end{prop}
\begin{proof}
By the Proposition \ref{intertwine-infty-coproduct}, we have that:
\begin{align}
(\Delta\otimes1)\mc{R}^{\infty}=\mc{R}^{\infty}_{13}\mc{R}^{\infty}_{23},\qquad(\Delta\otimes1)R_{\infty}=(R_{\infty})_{13}(R_{\infty})_{23}
\end{align}
It is equivalent to prove that $R_{\infty}$ valued in $\text{End}(K(1)\otimes K(1))$ is the same as $\mc{R}^{\infty}(u)$ in $\text{End}(K(1)\otimes K(1))$. We are going to check the proposition on each fixed point basis.

Recall the fixed point basis formula for the generators $h_{0,\pm d}$:
\begin{equation}
\begin{aligned}
\langle\mu|h_{\pm d}|\lambda\rangle=\delta^{\mu}_{\lambda}u^{\pm d}\text{sign}(d)(\frac{1}{(1-q_1^{\pm r})(1-q_2^{\pm r})}-\sum_{\square=(i,j)\in\lambda}q_1^{\pm(i-1)k}q_2^{\pm(j-1)k})
\end{aligned}
\end{equation}

In this case we have that:
\begin{equation}
\begin{aligned}
&\langle\mu|\otimes\langle\lambda|R_{\infty}|\lambda\rangle\otimes|\mu\rangle\\
=&\exp(-\sum_{d=1}^{\infty}\frac{(q_1^{d/2}-q_1^{-d/2})(q_2^{d/2}-q_2^{-d/2})(q^{d/2}-q^{-d/2})u^d}{d}(\frac{1}{(1-q_1^{d})(1-q_2^{d})}-\sum_{\square=(i,j)\in\lambda}q_1^{(i-1)d}q_2^{(j-1)d})\times\\
&\times(\frac{1}{(1-q_1^{-d})(1-q_2^{-d})}-\sum_{\square'=(i',j')\in\mu}q_1^{-(i'-1)d}q_2^{-(j'-1)d}))\\
=&\exp(-\sum_{d=1}^{\infty}\frac{(1-q_1^{d})(1-q_2^{d})(1-q^{-d})u^d}{d}\frac{1}{(1-q_1^d)(1-q_2^d)(1-q_1^{-d})(1-q_2^{-d})})\times\\
&\times\prod_{\square'=(i',j')\in\mu}\exp(\sum_{d=1}^{\infty}\frac{(1-q_1^{d})(1-q_2^{d})(1-q^{-d})u^d}{d}\frac{1}{(1-q_1^d)(1-q_2^d)}q_1^{-(i'-1)d}q_2^{-(j'-1)d})\times\\
&\times\prod_{\square=(i,j)\in\lambda}\exp(\sum_{d=1}^{\infty}\frac{(1-q_1^{d})(1-q_2^{d})(1-q^{-d})u^d}{d}\frac{1}{(1-q_1^{-d})(1-q_2^{-d})}q_1^{(i-1)d}q_2^{(j-1)d})\times\\
&\times\prod_{\square=(i,j)\in\lambda}\prod_{\square'=(i',j')\in\mu}\exp(-\sum_{d=1}^{\infty}\frac{(1-q_1^{d})(1-q_2^{d})(1-q^{-d})u^d}{d}q_1^{(i-i')d}q_2^{(j-j')d})
\end{aligned}
\end{equation}

The first term is independent of the choice of the basis:
\begin{equation}
\begin{aligned}
N(u)=&\exp(-\sum_{d=1}^{\infty}\frac{(1-q_1^{d})(1-q_2^{d})(1-q^{-d})u^d}{d}\frac{1}{(1-q_1^d)(1-q_2^d)(1-q_1^{-d})(1-q_2^{-d})})\\
=&\exp(-\sum_{d=1}^{\infty}\frac{(1-q^{-d})}{(1-q_1^{-d})(1-q_2^{-d})}u^d)
\end{aligned}
\end{equation}

For the second term:
\begin{equation}
\begin{aligned}
&\prod_{\square'=(i',j')\in\mu}\exp(\sum_{d=1}^{\infty}\frac{(1-q_1^{d})(1-q_2^{d})(1-q^{-d})u^d}{d}\frac{1}{(1-q_1^d)(1-q_2^d)}q_1^{-(i'-1)d}q_2^{-(j'-1)d})\\
=&\prod_{\square'=(i',j')\in\mu}\exp(\sum_{d=1}^{\infty}\frac{(1-q^{-d})u^d}{d}q_1^{-(i'-1)d}q_2^{-(j'-1)d})=\prod_{\square'=(i',j')\in\mu}\frac{(1-uq_1^{-i'+1}q_2^{-j'+1})}{(1-uq_1^{-i'}q_2^{-j'})}
\end{aligned}
\end{equation}

For the third term:
\begin{equation}
\begin{aligned}
&\prod_{\square=(i,j)\in\lambda}\exp(\sum_{d=1}^{\infty}\frac{(1-q_1^{d})(1-q_2^{d})(1-q^{-d})u^d}{d}\frac{1}{(1-q_1^{-d})(1-q_2^{-d})}q_1^{(i-1)d}q_2^{(j-1)d})\\
=&\prod_{\square=(i,j)\in\lambda}\exp(-\sum_{d=1}^{\infty}\frac{(1-q^d)u^d}{d}q_1^{(i-1)d}q_2^{(j-1)d})=\prod_{\square=(i,j)\in\lambda}\frac{(1-uq_1^iq_2^j)}{(1-uq_1^{i-1}q_2^{j-1})}
\end{aligned}
\end{equation}

For the last term:
\begin{equation}
\begin{aligned}
&\prod_{\square=(i,j)\in\lambda}\prod_{\square'=(i',j')\in\mu}\exp(-\sum_{d=1}^{\infty}\frac{(1-q_1^{d})(1-q_2^{d})(1-q^{-d})u^d}{d}q_1^{(i-i')d}q_2^{(j-j')d})\\
=&\prod_{\square=(i,j)\in\lambda}\prod_{\square'=(i',j')\in\mu}\frac{(1-uq_1^{i-i'+1}q_2^{j-j'})(1-uq_1^{i-i'}q_2^{j-j'+1})(1-uq_1^{i-i'-1}q_2^{j-j'-1})}{(1-uq_1^{i-i'-1}q_2^{j-j'})(1-uq_1^{i-i'}q_2^{j-j'-1})(1-uq_1^{i-i'+1}q_2^{j-j'+1})}
\end{aligned}
\end{equation}

Combining the result, we have that:
\begin{equation}
\begin{aligned}
R^{\infty}(u)=&N(u)\prod_{\square'=(i',j')\in\mu}\frac{(1-uq_1^{-i'+1}q_2^{-j'+1})}{(1-uq_1^{-i'}q_2^{-j'})}\prod_{\square=(i,j)\in\lambda}\frac{(1-uq_1^iq_2^j)}{(1-uq_1^{i-1}q_2^{j-1})}\times\\
&\times\prod_{\square=(i,j)\in\lambda}\prod_{\square'=(i',j')\in\mu}\frac{(1-uq_1^{i-i'+1}q_2^{j-j'})(1-uq_1^{i-i'}q_2^{j-j'+1})(1-uq_1^{i-i'-1}q_2^{j-j'-1})}{(1-uq_1^{i-i'-1}q_2^{j-j'})(1-uq_1^{i-i'}q_2^{j-j'-1})(1-uq_1^{i-i'+1}q_2^{j-j'+1})}
\end{aligned}
\end{equation}

On the other hand, the geometric $R$-matrix of the infinite slope is given by:
\begin{equation}
\begin{aligned}
\mc{R}^{\infty}(u)=\frac{\wedge^*(\frac{q_1V'}{V''})\wedge^*(\frac{q_2V'}{V''})\wedge^*(\frac{q^{-1}V'}{V''})}{\wedge^*(\frac{q_1^{-1}V'}{V''})\wedge^*(\frac{q_2^{-1}V'}{V''})\wedge^*(\frac{qV'}{V''})}\frac{\wedge^*(\frac{V'}{W''})}{\wedge^*(\frac{q^{-1}V'}{W''})}\frac{\wedge^*(\frac{qW'}{V''})}{\wedge^*(\frac{W'}{V''})}
\end{aligned}
\end{equation}

If we use the localisation formula of the tautological bundle in \ref{character-formula-tauto}, we can show that $\mc{R}^{\infty}(u)$ coincides with $R^{\infty}(u)$ up to a function $N(u)$.

\end{proof}

Now we can state the following theorem:
\begin{thm}
There is an isomorphism of Hopf algebras:
\begin{align}
U_{q_1,q_2}(\hat{\hat{\mf{gl}}}_1)\cong U_{q}(\hat{\mf{g}}_{Q}^{MO})
\end{align}
up to some centres.
\end{thm}
\begin{proof}
First we prove that it is an inclusion of Hopf algebras of arbitrary coproduct $\Delta_{a/b}$. This follows from the fact that $R_{a/b}^{\pm}=R_{a/b}^{\pm,MO}$ and the fact that $U_{q,t}(\hat{\hat{\mf{gl}}}_1)\hookrightarrow U_{q}(\hat{\mf{g}}_{Q}^{MO})$ is the inclusion of Hopf algebras of coproduct $\Delta_{\infty}$ from the main theorem in \cite{N23} and the Proposition \ref{intertwine-infty-coproduct}. Moreover, since the geometric $R$-matrix $\mc{R}^{s}$ coincides with the universal $R$-matrix $R^{s}$ evaluated in the modules $K(r_1)\otimes K(r_2)$ up to a constant $N(u)$, this means that these two algebras are isomorphic up to a centre. Thus the proof is finished.

\end{proof}

\end{document}